\documentclass[preprint]{imsart}

\RequirePackage[OT1]{fontenc}
\RequirePackage{amsthm,amsmath,latexsym}
\RequirePackage[numbers]{natbib}
\RequirePackage[colorlinks,citecolor=blue,urlcolor=blue]{hyperref}

\usepackage{amsfonts}

\startlocaldefs
\numberwithin{equation}{section}
\theoremstyle{plain}

\newtheorem{definition}{Definition}[section]
\newtheorem{proposition}[definition]{Proposition}
\newtheorem{theorem}[definition]{Theorem}
\newtheorem{corollary}[definition]{Corollary}
\newtheorem{remark}[definition]{Remark}
\newtheorem{lemma}[definition]{Lemma}

\endlocaldefs

\newcommand{\inner}[2]{\langle #1 | #2 \rangle} 

\def\argmin{\mathop{\rm argmin}}

\newcommand{\cX}{{\cal X}}

\def\cY{{\cal Y}}

\makeatletter
\newcommand{\lleq}{\mathrel{\mathpalette\gl@align<}}
\newcommand{\ggeq}{\mathrel{\mathpalette\gl@align>}}
\newcommand{\gl@align}[2]{
\vbox{\baselineskip\z@skip\lineskip\z@
\ialign{$\m@th#1\hfil##\hfil$\crcr#2\crcr{}_{{}_{(=)}}\crcr}}}
\makeatother

\def\Label#1{\label{#1}\ [\ \text{#1}\ ]\ }
\def\Label{\label}

\newenvironment{proofof}[1]{\vspace*{5mm} \par \noindent
         \quad{\it Proof of #1:\hspace{2mm}}}{\endproof
\hfill$\Box$ \vspace*{3mm}
}

\begin{document}

\begin{frontmatter}
\title{Finite-length Analysis on 
Tail probability for Markov Chain and Application to Simple Hypothesis Testing}
\runtitle{Finite-length Analysis for Markov Chains}

\begin{aug}
\author{{Shun Watanabe$^{1}$ 
and Masahito Hayashi$^{2}$}\\
\rm $^1$Department of  Information Science and Intelligent Systems, 
University of Tokushima, Japan, \\
and Institute for System Research, University of Maryland, College Park. \\
$^2$Graduate School of Mathematics, Nagoya University, Japan, \\
and Centre for Quantum Technologies, National University of Singapore, Singapore. \\
}
\runauthor{S. Watanabe and M. Hayashi}
\end{aug}

\begin{abstract}
Using terminologies of information geometry,
we derive upper and lower bounds of the tail probability of the sample mean.
Employing these bounds, we obtain upper and lower bounds of the minimum error probability of the 2nd kind of error under the exponential constraint for the error probability of the 1st kind of error
in a simple hypothesis testing for a finite-length Markov chain,
which yields the Hoeffding type bound.
For these derivations, we derive upper and lower bounds of
cumulant generating function for Markov chain.
As a byproduct, we obtain another simple proof of central limit theorem
for Markov chain.
\end{abstract}



\end{frontmatter}

\section{Introduction}\Label{s1}
Markov chain is a natural model for probability distribution with stochastic correlation.
Under this model, we often focus on the sample mean of $n$ observations,
and discuss the cumulant generating function and the tail probability.
Many existing studies investigated their asymptotic 
behaviors \cite{DZ,CLT2,CLT3,CLT4,MD}.
For example, 
the papers \cite{CLT2,CLT3,CLT4} showed 
the central limit theorem, i.e.,
they proved that the difference between the sample mean and the expectation asymptotically obeys the Gaussian distribution. 
Dembo and Zeitouni \cite{DZ} 
derived the asymptotic cumulant generating function 
and the large deviation bound by using its Legendre transform.
Further, other existing studies \cite{N,NK} 
investigated the simple hypothesis testing for Markov chains.
They derived the Hoeffding bound \cite{Hoeffding} for two Markov chains, 
i.e., 
the exponentially decreasing rate of the second error probability
under the exponential constraint for the first error probability.
In the independently and identically distributed (i.i.d.) case,
as the generalization of Stein's lemma.
Strassen \cite{Strassen} derived the asymptotic expansion of the exponential decreasing rate of the second error probability up to the order $\sqrt{n}$, 
under the constant constraint for the first error probability,
whose quantum extension was recently done by the papers in \cite{Tomamichel,Li}.

Indeed, it is not difficult to give a bound 
when it is not so tight or its computation is not so easy.
Here, we should mention a proper requirement for a better finite-length bound as follows.
\begin{itemize}
\item[(1)] Asymptotic tightness.
For example, in the case of the tail probability,
the bound can recover one of the following in the limit $n \to \infty$;
\begin{itemize}
\item[(T1)] Central limit theorem \cite{CLT2,CLT3,CLT4}
\item[(T2)] Moderate deviation \cite{MD,CLT2}
\item[(T3)] Large deviation \cite{DZ,CLT2}
\end{itemize}
\item[(2)] Computability. The bound should have 
less computational complexity, e.g., $O(1)$, $O(n)$ or $O(n \log n)$.
For example, we call the bound {\it $O(1)$-computable} 
when its computation complexity is $O(1)$.
\end{itemize}
In the i.i.d. case, it is known that 
the Markov inequality derives an upper bound of the tail probability 
that attains the asymptotic tightness in the sense of (T2) and (T3)
and is called Chernoff bound \cite{DZ,Massart}.
However, even in the i.i.d. case, there is no $O(1)$-computable finite-length lower bound 
that attains the asymptotic tightness in the sense of (T2) nor (T3).
The Berry-Essen theorem gives 
upper and lower $O(1)$-computable bounds of the tail probability 
that attain the asymptotic tightness in the sense of (T1) in the i.i.d. case (see e.g., \cite{Berry}).
The paper \cite[Theorem 2]{Herve} extended the Berry-Essen theorem to the Markov chain,
and gave similar upper and lower $O(1)$-computable bounds for the Markov chain.

In the case of simple hypothesis testing,
three kinds of the asymptotic tightness are characterized as follows.
\begin{itemize}
\item[(H1)] Constant constraint for the first error probability
$\epsilon=const$.
\item[(H2)] Moderate deviation type constraint for the first error probability $\epsilon=e^{-n^{1-2t}r}$ with $t \in (0,\frac{1}{2})$.
\item[(H3)] Large deviation type constraint for the first error probability $\epsilon=e^{-nr}$ (Hoeffding bound \cite{N,NK}).
\end{itemize}
In the i.i.d. case (including the quantum case),
the paper \cite{Tomamichel}
derived lower and upper $O(1)$-computable finite-length bounds 
for the second error probability 
that attain the asymptotic tightness in the sense of (H1). 
Also, it is not difficult to derive 
an upper $O(1)$-computable finite-length bound for the second error probability 
that attains the asymptotic tightness in the sense of (H2) nor (H3).
However, 
no study addressed 
a lower $O(1)$-computable finite-length bound for the second error probability 
that attains the asymptotic tightness in the sense of (H2) nor (H3)
even in the i.i.d. case.

This paper derives the finite-length bounds for the above topics
satisfying the above requirement.
Firstly, we derive upper and lower bounds of 
the cumulant generating function when $n$ observations are given.
We show that these limits recover
the asymptotic cumulant generating function \cite{DZ}.
Using our evaluation of the cumulant generating function, 
we also derive upper and lower $O(1)$-computable bounds of the tail probability
that attains the asymptotic tightness in the sense of (T2) and (T3)
in the Markov chain as well as in the i.i.d. case.
Our analysis covers 
the sample mean of two-input functions like $g(X_{k+1},X_k)$
as well as the simple sample mean $\sum_{i=1}^n \frac{X_i}{n}$.
As a byproduct, employing the evaluation of the cumulant generating function, 
we simply reproduce the central limit theorem \cite{CLT2,CLT3,CLT4}.
Indeed, since we address a general function $g(X_{k+1},X_k)$,
our evaluations can be applied to the sample mean of the hidden Markov random variables.

For the simple hypothesis testing,
this paper derives the lower and upper $O(1)$-computable bounds of the second error probability under the same constraint with finite observations
whose limits recover the asymptotic bound (H3)\cite{N,NK}
and the asymptotic bound (H2).
For describing these finite-length bounds, 
we employ the notations given by the transition matrix version of information geometry, i.e.,
the relative entropy (Kullback Leibler divergence), 
the relative R\'{e}nyi entropy,
exponential family, 
natural parameter,
and expectation parameter \cite{NK,HN,HW14-1}.
Further, 
employing the Markov version of the Berry-Essen theorem \cite[Theorem 2]{Herve}, 
we also obtain another type $O(1)$-computable finite-length bound,
which derives the asymptotic bound (H1)
as a generalization of the result by Strassen \cite{Strassen}.

Indeed, there are two ways to define a transition matrix version of
exponential family.
We employ the definition by \cite{NK,HN,HW14-1}, which is different from the definition by \cite{Feigin,Hudson,Bhat,Bhat2,Stefanov,Kuchler-Sorensen,Sorensen}.
The exponential family to be used plays an essential role in our derivation.
That is, the exponential family 
enables us to discuss simple hypothesis testing and the parameter estimation \cite{HW14-1} in a unified manner.
The obtained bounds are used for the evaluations of several 
information theoretical problems \cite{W-H2}.

The remaining of this paper is organized as follows.
Section \ref{s2} gives the brief summary of obtained results.
In Section \ref{s4}, we review 
an exponential family of transition matrices \cite{NK,HN,HW14-1}
in the one-parameter case.
In Section \ref{s11}, 
we characterize Legendre transform of the potential function.
In Section \ref{s12}, we give useful upper and lower bounds of the cumulant generating function.
In Section \ref{s13}, we give a simple alternative proof of the central limit theorem of Markov chain.
In Section \ref{s14}, we also give useful upper and lower bounds of the tail probability with finite observation, 
which produces the large deviation bound of the tail probability.
In Section \ref{s15},
using these bounds, we derive upper and lower bounds of
the second error probability of simple hypothesis testing,
which yields the Hoeffding type bounds.

\section{Summary of results}\Label{s2}
Here, we prepare notations and definitions.
For a given transition matrix $W$ over $\cX$,
we define 
$W^{\times n}(x_n,x_{n-1}, \ldots,x_1|\bar{x}):=
W(x_n|x_{n-1})W(x_{n-1}|x_{n-2})\cdots W(x_1|\bar{x})$
and 
$W^n(x|\bar{x})= \sum_{x_{n-1}, \ldots,x_1} W^{\times n}(x,x_{n-1}, \ldots,x_1|\bar{x})$.
For a given distribution $P$ on $\cX$
and a transition matrix $V$ from $\cX$ to $\cY$,
we define $V \times P(y,x):=V(y|x)P(x)$.
and $V P(y):=\sum_x V \times P(y,x)$.

A non-negative matrix $W$ is called {\it irreducible} 
when for each $x,\bar{x}\in \cX$, there exists a natural number $n$ such that $W^n(x|\bar{x})>0$ \cite{MU}.
An irreducible matrix $W$ is called {\it ergodic} when 
there are no input $\bar{x}$ and no integer $n'$
such that $W^{n} (\bar{x}|\bar{x})=0$ unless 
$n$ is divisible by $n'$ \cite{MU}.
It is known that 
the output distribution of $W^n P$ converges to the stationary distribution of $W$
for a given ergodic transition matrix $W$ \cite{kemeny-snell-book,MU}.

\subsection{Cumulant generating function}
Assume that the random variable $X_n$ obeys the Markov process with the 
transition matrix $W(x|\bar{x})$.
In this paper, for a two-input function $g(x,\bar{x})$,
we focus on the random variable 
$g^n(X^{n+1}):= \sum_{i=1}^n g(X_{i+1},X_{i})$,
and $X^{n+1}:= (X_{n+1}, \ldots,X_1)$.
This is because a two-input function $g(x,\bar{x})$ is closely related to an exponential family of transition matrices.
Indeed, the simple sample mean can be treated in the formulation by choosing $g(x,\bar{x})$ as $x$ or $\bar{x}$.
Here, when we choose a general function $g(x)$,
$g^n(X^{n+1})= \sum_{i=1}^n g(X_{i+1})$ is the sample mean of the hidden Markov random variable.
So, our results can applied to the hidden Markov random case.

We 
denote the Perron-Frobenius eigenvalue of 
$W(x|\bar{x})e^{\theta g(x,\bar{x})}$ by $\lambda_\theta$
and define the potential function 
$\phi(\theta):= \log \lambda_\theta$.
Then, we focus on the cumulant generating function 
$\phi_n(\theta):=\log \mathsf{E}[e^{\theta g^n(X^{n+1})}]$,
where $\mathsf{E}$ denotes the expectation.
We will define 
functions $\underline{\delta}(\theta)$ and $\overline{\delta}(\theta)$ in Section \ref{s12}
so that $\underline{\delta}(\theta)\to 0$ and $\overline{\delta}(\theta) \to 0$ as $\theta\to 0$.
Then, we will evaluate $\phi_n(\theta)$ as
\begin{align}
n \phi(\theta)+\underline{\delta}(\theta) \le \phi_n(\theta) 
\le n \phi(\theta)+\overline{\delta}(\theta).
\end{align}

\subsection{Tail probability}
Given an irreducible and ergodic transition matrix $W$,
we will evaluate the tail probability of the random variable $g^n(X^{n+1})$
by using the one-parameter exponential family $W_\theta$ given in \cite[Section 3]{HW14-1}
and 
the relative entropies $D(W_{\theta}\|W_{\bar{\theta}})$
and $D_{1+s}(W_\theta\|W_{\bar{\theta}})$ explained in Section \ref{s4}.
as follows.
For any $a > \mathsf{E}[g]$, we will show
\begin{eqnarray}
- \log P \{ \tilde{g}^n(X^{n+1}) \ge n a \} \ge 
 n D(W_{{\phi'}^{-1}(a)} \| W_{0} )- \overline{\delta}(\theta),\Label{27-32b}
\end{eqnarray}
where ${\phi'}^{-1}(a)$ is the inverse function of 
$\phi'(\theta)=\frac{d \phi}{d \theta}(\theta)$, 
i.e., 
$\frac{d \phi}{d \theta}({\phi'}^{-1}(a))=a$.
Conversely, we will show
\begin{align}
\nonumber
\lefteqn{ - \log P \{ \tilde{g}^n(X^{n+1}) \ge n a \}  }  \\
\le & 
 \inf_{s > 0 \atop \theta > {\phi'}^{-1}(a)} 
n D_{1+s}(W_\theta\|W_0)
+\frac{1}{s}[\overline{\delta}((1+s)\theta) - \underline{\delta}(\theta)] \nonumber \\
 &\quad -\frac{1+s}{s}
 \log \left(1- e^{-n D(W_{{\phi'}^{-1}(a)}\|W_{\theta})
+\overline{\delta}({\phi'}^{-1}(a)) -\underline{\delta}(\theta)
} \right) . 
   \Label{27-33b}
\end{align}

Similarly, for $a < \mathsf{E}[g]$, we will show
\begin{eqnarray}
- \log P \{ \tilde{g}^n(X^{n+1}) \le n a \} \ge  
 n D(W_{{\phi'}^{-1}(a)} \| W_0)- \overline{\delta}(\theta).
\end{eqnarray}
Conversely, we will show
\begin{align}
\nonumber
\lefteqn{ - \log P \{ \tilde{g}^n(X^{n+1}) \le n a \} }  \\
\le & 
\inf_{s > 0 \atop \theta < {\phi'}^{-1}(a)} 
n D_{1+s}(W_\theta\|W_0)
+\frac{1}{s}[\overline{\delta}((1+s)\theta) - \underline{\delta}(\theta)] \nonumber \\
 &\quad -\frac{1+s}{s}
 \log \left(1- e^{-n D(W_{{\phi'}^{-1}(a)}\|W_{\theta})
+\overline{\delta}({\phi'}^{-1}(a)) -\underline{\delta}(\theta)
} \right) . 
  \Label{eq:7-1} 
\end{align}

\subsection{Simple hypothesis testing}
Now, we consider the hypothesis testing with the two hypotheses
$W_0^{n-1}\times P_0$ and $W_1^{n-1}\times P_1$.
Then, we consider 
\begin{align}
\nonumber
&\beta_\epsilon(W_1^{n-1}\times P_1 \| W_0^{n-1}\times P_0)\\
:=&
\min_{S \subset \cX^{n+1}}
\{ 1-W_0^{n-1}\times P_0(S) | W_1^{n-1}\times P_1(S)\le \epsilon\}.
\Label{23-1}
\end{align}
Using the one-parameter exponential family $W_\theta$ of transition matrices with the generator
$g(x,\bar{x}):= \log \frac{W_1(x|\bar{x})}{W_0(x|\bar{x})}$,
the cumulant generating function $\phi(\theta)$ defined by $g(x,\bar{x})$,
we will show that
\begin{align}
\nonumber
& 
 \sup_{0 \le \theta \le 1} \frac{n (-\theta  r - \phi(\theta))-
\underline{\delta}(\theta) }{1-\theta}
 \\
\le & -\log \beta_{e^{-nr}}(W_1^{\times n}\times P_1 \| W_0^{\times n}\times P_0) \nonumber \\
\le & \inf_{s>0,\theta \in (\hat{\theta}(r),1)}
n D_{1+s}(W_{\theta}\|W_0) 
+\frac{1}{s}( \overline{\delta}( (1+s) \theta) -(1+s) \underline{\delta}(\theta))\nonumber \\
&\quad- \frac{1+s}{s} \log (1- 2 e^{-n D(W_{\hat{\theta}(r)}\|W_{\theta})
-\underline{\delta}(\theta)
+\frac{(1-\theta) \overline{\delta}(\hat{\theta}(r)) }{1-\hat{\theta}(r)} }) 
,
\Label{27-5b} 
\end{align}
where
the functions 
$\hat{\theta}(r)$ is given in Section \ref{s11}.
We will also asymptotically characterize 
$\beta_\epsilon(W_1^{n-1}\times P_1 \| W_0^{n-1}\times P_0)$ with a fixed $\epsilon$.

\section{Geometric structure for transition matrices}\Label{s4}
In this section, 
we review the definition and the properties 
of the one-parameter exponential family of transition matrices \cite{NK,HN}
by following the logical order of \cite[Section 4]{HW14-1}
although a large part of results for exponential family of transition matrices
were obtained by Nagaoka \cite{HN} and Nakagawa and Kanaya \cite{NK}.
This is because the logical order of \cite[Section 4]{HW14-1}
is more suitable for the context of this paper. 
These relations are explained in \cite[Remarks 3.5, 4.12, and 4.14]{HW14-1}.
Note that the definition of exponential family in this paper is 
different from that by the papers \cite{Feigin,Hudson,Bhat,Bhat2,Stefanov,Kuchler-Sorensen,Sorensen}
as explained in \cite[Remark 4.13]{HW14-1}.

\subsection{Preparations}
For the definition and the properties 
of the one-parameter exponential family of transition matrices,
we prepare the following things.

\begin{lemma}\cite[Lemma 3.1]{HW14-1} 
\Label{L1}
Consider an irreducible and ergodic transition matrix $W$ over $\cX$
and a real-valued function $g$ on $\cX \times \cX$.
Then, we define the support 
$\cX^2_W:=\{(x,\bar{x}) \in \cX^2| W(x|\bar{x})>0\}$.
Define $\phi(\theta)$ as the logarithm of the Perron-Frobenius eigenvalue of the matrix:
\begin{align}
\tilde{W}_{\theta}(x|\bar{x}):= W(x|\bar{x}) e^{\theta g(x,\bar{x})}.
\end{align}
Then, the function $\phi(\theta)$ is convex.
Further, the following conditions are equivalent.
\begin{itemize}
\item[(1)] 
No real-valued function $f$ on $\cX$ satisfies that 
$g(x,\bar{x})=f(x)-f(\bar{x})+c$ for any $(x,\bar{x})\in \cX^2_W$ with a constant $c \in \mathbb{R}$.
\item[(2)]
The function $\phi(\theta)$ is strictly convex, i.e., 
$\frac{d^2 \phi}{d \theta^2}(\theta)>0$ for any $\theta$.
\item[(3)]
$\frac{d^2 \phi}{d \theta^2}(\theta)|_{\theta=0}>0$.
\end{itemize}
\end{lemma}

Using Lemma \ref{L1},
given two distinct ergodic transition matrices $W$ and $V$
with the same support,
we define the relative entropy and the relative R\'{e}nyi entropies. 
For this purpose, we denote the logarithm of
the Perron-Frobenius eigenvalue of the matrix
$W(x|\bar{x})^{1+s}V(x|\bar{x})^{-s}$ by $\varphi(1+s)$.
Then, we define 
\begin{align}
D({W} \| V):= 
\frac{d \varphi}{d s}(1) \Label{1-10},\quad
D_{1+s}({W} \| V):=
\frac{\varphi(1+s)}{s}.
\end{align}
Note that the limit $\lim_{s \to 0}D_{1+s}({W} \| V)$ equals
$D({W} \| V)$.
Since
$W$ and $V$ are distinct,
the function $\log \frac{W(x|\bar{x})}{V(x|\bar{x})}$ satisfies the condition for the function $g$ in Lemma \ref{L1}.
Hence, the function $s \mapsto s D_{1+s}({W} \| V)$ is strictly convex,
which implies that
$s D_{1+s}({W} \| V) <
(1-\frac{\bar{s}}{s})0+ \frac{s}{\bar{s}}\bar{s} D_{1+\bar{s}}({W} \| V)$
for $0<s<\bar{s}$.
Since a similar relation holds for $0>s>\bar{s}$, 
the relative R\'{e}nyi entropy $D_{1+s}({W} \| V)$ is strictly monotone increasing with respect to $s$.

\subsection{Exponential family}
Now, we focus on a transition matrix $W(x|\bar{x})$ from $\cX$ to $\cX$
and a real-valued function $g$ on $\cX \times \cX$
satisfying the condition in Lemma \ref{L1}.
In the following, we assume that the function $g$ satisfies condition in Lemma \ref{L1}.
Then, we will define the matrix ${W}_{\theta}(x|\bar{x})$ from $\cX$ to $\cX$ for $\theta$
by following steps below.
For this purpose, we define the matrix $\tilde{W}_{\theta}(x|\bar{x})$ from $\cX$ to $\cX$ by 
\begin{align}
\tilde{W}_{\theta}(x|\bar{x}):= W(x|\bar{x}) e^{\theta g(x,\bar{x})}.
\Label{5-6}
\end{align}
Using the Perron-Frobenius eigenvalue $\lambda_{\theta}$ of $\tilde{W}_{\theta}$,
we define the potential function 
\begin{align}
\phi(\theta):=\log \lambda_{\theta}.\Label{4-1-2}
\end{align} 
Due to Lemma \ref{L1}, the second derivative 
$\frac{d^2 \phi}{d \theta^2}$ is strictly positive.
Hence, the potential function $\phi(\theta)$ is strictly convex.
In the following, using the strictly convex function $\phi(\theta)$,
we define a one-parameter exponential family for transition matrices.


Note that, since the value $\sum_{x}\tilde{W}_{\theta}(x|\bar{x})$ generally depends on $\bar{x}$, 
we cannot make a transition matrix by simply multiplying a constant with the matrix $\tilde{W}_{\theta}$.
To make a transition matrix from the matrix $\tilde{W}_{\theta}$, we recall that
a non-negative matrix $V$ from $\cX$ to $\cX$ is a transition matrix
if and only if the vector $(1, \ldots,1)^T$ is an eigenvector of the transpose $V^T$.
In order to resolve this problem, we focus on the structure of the matrix $\tilde{W}_{\theta}$.
We denote the Perron-Frobenius eigenvectors
of $\tilde{W}_{\theta}$ and its transpose $\tilde{W}_{\theta}^T$
by $\tilde{P}_{\theta}$ and $\hat{P}_{\theta}$.
Since the irreducibility of $W$ guarantees 
the irreducibility of $\tilde{P}_{\theta}$,
the relation $\hat{P}_{\theta}(x) >0$ holds. 
According to \cite{NK,HN,CLT2,HW14-1}\footnote{Appendix of \cite{HW14-1} explains detailed relation the papers \cite{NK,HN,CLT2,HW14-1}
for an exponential family of transition matrices.}, 
we define the matrix ${W}_{\theta}(x|\bar{x})$ as
\begin{align}
{W}_{\theta}(x|\bar{x}):= \lambda_{\theta}^{-1} \hat{P}_{\theta}(x)
\tilde{W}_{\theta}(x|\bar{x})\hat{P}_{\theta}(\bar{x})^{-1}.
\end{align}
The matrix ${W}_{\theta}(x|\bar{x})$ is a transition matrix 
because vector $(1, \ldots,1)^T$ is an eigenvector of the transpose $W_{\theta}^T$.
In the following, we call the family of transition matrices 
${\cal E}:=\{ {W}_{\theta} \}$ an {\it exponential family} of transition matrices with the generator $g$.

Using the potential function $\phi(\theta)$,
we explain several concepts 
for transition matrices based on Lemma \ref{L1}, formally.
We call the parameter $\theta$ the natural parameter, and the parameter $\eta(\theta):= 
\phi'(\theta)=\frac{d \phi}{d \theta}(\theta)$ the expectation parameter.
For $\eta$,
we define the inverse function ${\phi'}^{-1}(\eta)$ of
$\phi'$ as 
\begin{align}
\phi'({\phi'}^{-1}(\eta))=\eta.\Label{4-1-1}
\end{align}
Then, we define the Fisher information for the natural parameter by 
the second derivative 
$\frac{d^2 \phi}{d \theta^2}(\theta)$.
The Fisher information for the expectation parameter 
is given as 
$\frac{d^2 \phi}{d \theta^2}(\theta)^{-1}$.

\begin{lemma}\cite[Lemma 4.4]{HW14-1}\Label{L7}
The relative entropy and the relative R\'{e}nyi entropies between two transition matrices 
${W}_{\theta}$ and ${W}_{\bar{\theta}}$ are characterized as
\begin{align}
D({W}_{\theta} \| {W}_{\bar{\theta}})= &
(\theta-{\bar{\theta}})\frac{d \phi}{d \theta}(\theta)- \phi(\theta)+ \phi(\bar{\theta}) \Label{1-1}\\
D_{1+s}({W}_{\theta} \| {W}_{\bar{\theta}})=&
\frac{\phi((1+s)\theta -s \bar{\theta})-(1+s) \phi(\theta) +s \phi(\bar{\theta})}{s}.
 \Label{1-2}
\end{align}
\end{lemma}

In the following,
$\mathsf{E}_{W}$ denotes the expectation with respect to 
the joint distribution 
when the conditional distribution is given by the transition 
matrix $W$ and the input distribution is given by
the stationary distribution of $W$.
Then, for a generator $g$ and a real number $a$,
we define ${\cal M}_{g,a}$ as
\begin{align}
{\cal M}_{g,a}:= \{W| \mathsf{E}_{W}g(X,X')= a\}.
\end{align}

A transition matrix version of the Pythagorean theorem \cite{AN} 
holds as follows.
\begin{theorem}
\cite[Lemma 5]{NK},\cite[Corollary 4.8]{HW14-1}
\Label{T5-1}
For a transition matrix $V$, a generator $g$, and a real number $a$,
we define 
\begin{align}
V^*:= \argmin_{W \in {\cal M}_{g,a}}D(W\|V).
\end{align}
(1) Any transition matrix $W \in {\cal M}_{g,a}$ satisfies
\begin{align}
D(W\|V)=D(W\|V^*)+D(V^*\|V).
\end{align}
(2) The transition matrix $V^*$ is the intersection of 
the set ${\cal M}_{g,a}$ and the exponential family generated by $g$
containing $V$.
\end{theorem}

Due to Lemma \ref{L7}, the Fisher information 
$\frac{d^2\phi}{d\theta^2}(\theta_0)$
can be characterized by the limits of the relative entropy and relative R\'{e}nyi entropy as follows.
\begin{lemma}\Label{L20}
Under the limit $\delta \to 0$,
we have
\begin{align} 
\lim_{\delta\to 0}\frac{1}{\delta^2} D(W_{\theta_0+\delta}\|W_{\theta_0})
&=
\lim_{\delta\to 0}\frac{1}{\delta^2}
 D(W_{\theta_0}\|W_{\theta_0+\delta})
=
\frac{1}{2} \frac{d^2\phi}{d\theta^2}(\theta_0) ,
\Label{27-20} 
\\
\lim_{\delta\to 0}\frac{1}{\delta^2} 
D_{1+s}(W_{\theta_0+\delta}\|W_{\theta_0})
&=
\lim_{\delta\to 0}\frac{1}{\delta^2}
 D_{1+s}(W_{\theta_0}\|W_{\theta_0+\delta})
\nonumber
\\
&=
\frac{1+s}{2} \frac{d^2\phi}{d\theta^2}(\theta_0). 
\Label{27-21} 
\end{align}
\end{lemma}

\section{Relation with Legendre transform}\Label{s11}
Given two irreducible and ergodic transition matrices $W$ and $V$,
we choose 
the exponential family $W_{\theta}$ with the generator $g(x,\bar{x}):= \log V(x|\bar{x})- \log W(x|\bar{x})$ so that $W_{0}=W $ and $W_{1}=V$. 
In fact, an arbitrary exponential family $W_{\theta}$ can be written as the above form by choosing two irreducible and ergodic transition matrices as $W:=W_0 $ and $V:=W_1$.
The Legendre transform $\sup_{\theta \ge 0}[ \theta a - \phi(\theta) ] $ of the convex function $\phi$
can be characterized as follows.

\begin{lemma}\Label{L23}
When $a > - D(W \| V)$,
\begin{align}
\nonumber 
& \inf_{s > 0 \atop \theta > {\phi'}^{-1}(a)} D_{1+s}(W_{\theta} \| W_{0} )
= \inf_{s > 0 \atop \theta > {\phi'}^{-1}(a)} \frac{\phi((1+s)\theta) - (1+s)\phi(\theta)}{s} \\
=& {\phi'}^{-1}(a) a - \phi({\phi'}^{-1}(a)) 
= \sup_{\theta \ge 0}[ \theta a - \phi(\theta) ] 
=  D(W_{{\phi'}^{-1}(a)} \| W_{0} ) .
\Label{26-1}
\end{align}
Similarly, when $a < - D(W \| V)$,
\begin{align*}
&\inf_{s > 0 \atop \theta < {\phi'}^{-1}(a)} \frac{\phi((1+s)\theta) - (1+s)\phi(\theta)}{s} 
= {\phi'}^{-1}(a) a - \phi({\phi'}^{-1}(a)) \\
=& \sup_{\theta \le 0}[ \theta a - \phi(\theta) ] 
=  D(W_{{\phi'}^{-1}(a)} \| W_{0} ) .
\end{align*}
\end{lemma}

\begin{proof}
Since the function $\theta \mapsto \phi(\theta)$ is convex,
the function $s \mapsto \frac{\phi((1+s)\theta) - (1+s)\phi(\theta)}{s} $
is monotone increasing due to Lemma \ref{L11-21}.
Hence, 
$ \inf_{s > 0 } 
\frac{\phi((1+s)\theta) - (1+s)\phi(\theta)}{s} 
= \theta \frac{d \phi}{d\theta}(\theta) - \phi(\theta) $ for $\theta > {\phi'}^{-1}(a)$.
Thus,
\begin{align*}
& \inf_{s > 0 \atop \theta > {\phi'}^{-1}(a)} 
\frac{\phi((1+s)\theta) - (1+s)\phi(\theta)}{s} 
= \inf_{\theta > {\phi'}^{-1}(a)} \theta \frac{d \phi}{d\theta}(\theta) - \phi(\theta) \nonumber .
\end{align*}
Since the function $\theta \mapsto \phi(\theta)$ is convex,
the function $\theta \mapsto\theta \frac{d \phi}{d\theta}(\theta) - \phi(\theta) $
is monotone increasing for $\theta \ge 0$.
Therefore,
\begin{align*}
\inf_{\theta > {\phi'}^{-1}(a)} \theta \frac{d \phi}{d\theta}(\theta) - \phi(\theta) 
= {\phi'}^{-1}(a) a - \phi({\phi'}^{-1}(a)) 
= \sup_{\theta \ge 0}[ \theta a - \phi(\theta) ] ,
\end{align*}
where the second equation follows from the convexity of $\phi(\theta)$.
The final equation in (\ref{26-1}) is shown by
$ D(W_{{\phi'}^{-1}(a)} \| W_{0} ) = {\phi'}^{-1}(a) a - \phi({\phi'}^{-1}(a)) $.
\end{proof}

Now, for an arbitrary convex function $\phi$
and
$r>0$, we define the function $
\hat{\theta}(r)=\hat{\theta}[\phi](r)$ as the smaller solution of 
the equation
\begin{align}
(\theta-1)\frac{d \phi}{d\theta}(\theta)-\phi(\theta)
=D(W_{\theta}\|W_1)=r
\Label{12-31-1}
\end{align}
with respect to $\theta$.
Hence, due to 
the convexity of $\phi$, we have
\begin{align}
\nonumber 
& \inf_{s>0,\theta \in (0,\hat{\theta}(r))}
D_{1+s}(W_{\theta}\|W_0)  \\
=& \inf_{s>0,\theta \in (0,\hat{\theta}(r))}
\frac{1}{s} [\phi( (1+s) \theta) -(1+s) \phi(\theta) ] \nonumber\\
=&  \inf_{s>0}
\frac{1}{s} [\phi( (1+s) \hat{\theta}(r)) -(1+s) \phi(\hat{\theta}(r)) ] 
= 
\hat{\theta}(r) \frac{d \phi}{d \theta}( \hat{\theta}(r)) -\phi(\hat{\theta}(r))
\nonumber \\
=& 
-\hat{\theta}(r) 
\frac{r+\phi(\hat{\theta}(r) )}{1-\hat{\theta}(r) }
 -\phi(\hat{\theta}(r)) 
= \frac{-\hat{\theta}(r)  r - \phi(\hat{\theta}(r))}{1-\hat{\theta}(r)} \nonumber \\
\stackrel{(a)}{=}& \sup_{0 \le \theta < 1} 
\frac{-\theta  r - \phi(\theta)}{1-\theta} 
= \sup_{0 \le \theta < 1} 
\frac{\theta ( -r + D_{1-\theta}(W_0\|W_1))}{1-\theta},
\Label{26-10}
\end{align}
which implies the following lemma.
Here, $(a)$ can be derived as follows.
Due to Lemma \ref{L11-21}, 
the maximum can be attained when \eqref{12-31-1} holds, i.e.,
$\theta=\hat{\theta}(r)$.
Hence, we have $(a)$.
 
\begin{lemma}\Label{L1-3-1}
When $0 \le r \le D(W_0\|W_1)$,
\begin{align*}
&\sup_{0 \le \theta \le 1} \frac{-\theta  r - \phi(\theta)}{1-\theta} 
= \sup_{0 \le \theta \le 1} 
\frac{\theta ( -r + D_{1-\theta}(W_0\|W_1))}{1-\theta} \\
=& \inf_{s>0,\theta \in (0,\hat{\theta}(r))}
D_{1+s}(W_{\theta}\|W_0)
= D(W_{\hat{\theta}(r)}\|W_0)
=\min_{W:D(W\|W_1) \le r} D(W\|W_0).
\end{align*}
\end{lemma}
Here,
when $\frac{\phi(\theta)}{1-\theta}$ is regard as a function of 
$\delta:= \frac{-\theta }{1-\theta}$, i.e., is described as 
$f(\delta)$,
$\sup_{0 \le \theta \le 1} \frac{-\theta  r - \phi(\theta)}{1-\theta}$
is given as the Legendre transform of $f$, i.e,
$\sup_{0 \le \theta \le 1} \delta r - f(\delta)$

\begin{proof}
The first and second equations follow from (\ref{26-10}).
The third equation follows from (\ref{26-10}) and the relation
$D(W_{\hat{\theta}(r)}\|W_0)
=\hat{\theta}(r) \frac{d \phi}{d \theta}( \hat{\theta}(r)) -\phi(\hat{\theta}(r)) $.
Now, we show the final equation.
We choose $W$ satisfying that $D(W\|W_1) \le r$.
We also choose $a$ such that $W \in {\cal M}_{g,a}$, which is defined in Theorem \ref{T5-1}.
Then, we denote the intersection of 
the set ${\cal M}_{g,a}$
and the exponential family $\{W_\theta\}$ by $W_{\bar{\theta}}$.
Theorem \ref{T5-1} implies that
$D(W_{\bar{\theta}} \| W_1) \le r$ and 
$D(W_{\bar{\theta}} \| W_0) \le D(W \| W_0)$.
Thus, we obtain
\begin{align}
\min_{W:D(W\|W_1) \le r} D(W\|W_0)
=
\min_{\theta:D(W_{\theta}\|W_1) \le r} D(W_{\theta}\|W_0).\Label{5-9}
\end{align}
Due to the condition $0 \le r \le D(W_1\|W_0)$,
the above value equals $D(W_{\hat{\theta}(r)}\|W_0)$.
\end{proof}


\section{Information processing inequality}\Label{s11-2}
Now, we introduce a condition for a transition matrix as follows.
A transition matrix $W$ on ${\cal X} \times {\cal Y}$ is called {\it non-hidden} for ${\cal X}$
when $W_X(x|x'):=\sum_{y \in {\cal Y}}W(x,y|x',y')$ does not depend on $y'\in {\cal Y}$.
When the Markov chain for $X$ and $Y$ generated by $W$ and $W$ satisfies the above condition,
the sequence for $X$ is also a Markov chain, not a hidden Markov chain. 
This is the reason of the name of ``non-hidden''.
As a transition matrix version of information processing inequality, we have the following theorem.
\begin{theorem}\Label{t-1-1}
When two transition matrices $W$ and $V$ 
on ${\cal X} \times {\cal Y}$ are  non-hidden for ${\cal X}$, the following hold 
for $s \in (-1,0)\cup (0,\infty)$:
\begin{align}
D(W\|V) \ge D(W_X\|V_X) , \quad
D_{1+s}(W\|V) \ge D_{1+s}(W_X\|V_X) \Label{1-1-1}.
\end{align}
\end{theorem}
\begin{proof}
For $s>0$, 
let $\lambda$ and $b=(b_{x,y})$ be the 
Perron-Frobenius eigenvalue
and eigenvector of the matrix $W(x,y|x',y')^{1+s} V(x,y|x',y')^{-s}$.
Since
the reverse H\"{o}lder inequality implies 
\begin{align*}
& \sum_{y} W(x,y|x',y')^{1+s}
V(x,y|x',y')^{-s}\\
\ge &
(\sum_{y} W(x,y|x',y')^{(1+s)/(1+s)})^{1+s}
(\sum_{y}V(x,y|x',y')^{-s/(-s)})^{-s} \\
=& W_X(x|x')^{1+s} V_X(x|x')^{-s},
\end{align*}
the number $c_{x} =\sum_{y}b_{x,y} $ satisfies
\begin{align*}
& \lambda c_{x} =\sum_{y} \lambda b_{x,y}
= \sum_{y} \sum_{x',y'}  W(x,y|x',y')^{1+s}
V(x,y|x',y')^{-s} b_{x',y'}\\
\ge &
\sum_{x',y'} W_X(x|x')^{1+s} V_X(x|x')^{-s} b_{x',y'}
=
\sum_{x'} W_X(x|x')^{1+s} V_X(x|x')^{-s} c_{x'},
\end{align*}
which implies that 
$\max_{x}\frac{\sum_{x'} W_X(x|x')^{1+s} V_X(x|x')^{-s} c_{x'}}{c_x} \le \lambda$.
Due to Collatz-Wielandt formula, 
$\lambda$ is larger than 
the Perron-Frobenius eigenvalue of the matrix $W_X(x|x')^{1+s} V_X(x|x')^{-s}$.
Hence, we obtain the second inequality of (\ref{1-1-1}) with $s >0$.

When $s \in (-1,0)$, replacing the role of the reverse H\"{o}lder inequality
by the H\"{o}lder inequality, we can show that
$\min_{x}\frac{\sum_{x'} W_X(x|x')^{1+s} V_X(x|x')^{-s} c_{x'}}{c_x} \ge \lambda$.
Due to the Perron-Frobenius theorem, $\lambda$ is smaller than 
the Perron-Frobenius eigenvalue of the matrix $W_X(x|x')^{1+s} V_X(x|x')^{-s}$.
Hence, we obtain the second inequality of (\ref{1-1-1}) with $s\in (-1,0)$.
Taking the limit $s \to 0$, we obtain the first inequality of (\ref{1-1-1}).
\end{proof}

Theorem \ref{t-1-1} can be regarded as a part of information processing inequality as follows.
In the case of 
the information processing inequality between two distributions $P$ and $P'$,
we compare the relative entropy between $P$ and $P'$ and the relative entropy between $VP$ and $VP'$ for a given transition matrix $V$.
Since the relative entropy between $P$ and $P'$ equals the relative entropy between $V \times P$ and $V \times P'$,
it is enough to compare 
the relative entropy between $V \times P$ and $V \times P'$ and the relative entropy between $VP$ and $VP'$.
The difference between these relative entropies can be characterized as 
existence or non-existence of the marginalization for the input system.
Therefore, the information processing inequality 
can be reduced to the information processing inequality with respect to the marginalization.
As inequalities given in Theorem \ref{t-1-1} give the relations
among the relative entropies before/after the marginalization,
they can be regarded as an information processing inequality.
Therefore, it can be expected that
Theorem \ref{t-1-1} 
will play roles of information processing inequality in information theory.

\section{Cumulant generating function}\Label{s12}
In the following, we consider the Markov chain $X^{n+1}=
(X_1, \ldots, X_n,X_{n+1})$
generated by the transition matrix $W_0$
and an arbitrary initial distribution $P_0$.
That is, the random variable $X^{n+1}$
is subject to the distribution $W_0^{\times n} \times P_0$.
We consider the random variable 
$\tilde{g}^n(X^{n+1}):= \sum_{i=1}^n g(X_{i+1},X_i)+h(X_1)$
for a function $h$ on $\mathbb{R}$.
Then, we define the cumulant generating function 
\begin{align}
\phi_{n}(\theta) :=& \log \mathsf{E}_0 [e^{\theta \tilde{g}^n(X^{n+1})}] ,
\end{align} 
where $\mathsf{E}_0$ denotes the expectation under the distribution $W_0^{\times n} \times P_0$.

\begin{lemma} \Label{lemma:mgf-finite-evaluation}
Let $v_\theta$ be the eigenvector of 
$\tilde{W}_\theta^T$ with respect to 
the Perron-Frobenius eigenvalue $\lambda_\theta$ such that $\min_x v_\theta(x) = 1$.
Let $w_\theta(x) := P_0(x) e^{\theta h(x)}$. Then, we have
\begin{align} \Label{25-12}
n \phi(\theta) + \underline{\delta}(\theta) \le \phi_n(\theta) \le  
n \phi(\theta) + \overline{\delta}(\theta),
\end{align}
where 
\begin{eqnarray}
\overline{\delta}(\theta) := \log \inner{v_\theta}{w_\theta}, \quad
\underline{\delta}(\theta) :=  \log \inner{v_\theta}{w_\theta} - \log \max_x v_\theta(x).
\Label{25-12b}
\end{eqnarray}
\end{lemma}
\begin{proof}
Let $u$ be the vector such that $u(x)=1$ for every $x\in {\cal X}$. From the definition of
$\phi_n(\theta)$, we have the following sequence of calculations:
\begin{align*}
&e^{\phi_n(\theta)} 
= \sum_{x_n,\ldots,x_1} P(x_1) \prod_{i=1}^n W(x_{i+1}|x_{i})  e^{\theta \sum_{i=1}^n g(x_{i+1},x_{i}) + h(x_1)} \nonumber \\
=& \inner{u}{\tilde{W}_\theta^{n} w_\theta} 
\le \inner{v_\theta}{\tilde{W}_\theta^{n} w_\theta} 
= \inner{(\tilde{W}_\theta^T)^{n} v_\theta}{w_\theta} 
= \lambda_\theta^{n} \inner{v_\theta}{w_\theta} 
= e^{n \phi(\theta)} \inner{v_\theta}{w_\theta},
\end{align*}
which implies the right hand side inequality of \eqref{25-12}. On the other hand, we have the following sequence of calculations:
\begin{eqnarray*}
&& e^{\phi_n(\theta)} 
= \inner{u}{\tilde{W}_\theta^{n} w_\theta} 
\ge \frac{1}{\max_x v_\theta(x)} \inner{v_\theta}{\tilde{W}_\theta^{n} w_\theta} \nonumber \\
&=& \frac{1}{\max_x v_\theta(x)} \inner{(\tilde{W}_\theta^T)^{n} v_\theta}{w_\theta} 
= \lambda_\theta^{n} \frac{\inner{v_\theta}{w_\theta}}{\max_x v_\theta(x)} 
= e^{n \phi(\theta)} \frac{\inner{v_\theta}{w_\theta}}{\max_x v_\theta(x)},
\end{eqnarray*}
which implies the left hand side inequality of \eqref{25-12}.
\end{proof}

By taking the limit in \eqref{25-12}  of Lemma \ref{lemma:mgf-finite-evaluation},
we have the following.
\begin{corollary}\cite[Theorem 3.1.1]{DZ} \label{theorem:mgf-large-deviation-single-letter}
For $\theta \in \mathbb{R}$, we have
\begin{eqnarray}
\lim_{n \to \infty} \frac{1}{n} \phi_n(\theta) = \phi(\theta).
\end{eqnarray}
\end{corollary}

\begin{lemma} \Label{L20f}
\begin{align}
\lim_{\theta\to 0} \overline{\delta}(\theta) = 0 ,\quad 
\lim_{\theta\to 0} \underline{\delta}(\theta) = 0 . \Label{28-5}
\end{align}
\end{lemma}
\begin{proof}
From the construction of $v_\theta$ and $w_\theta$,
the vectors $v_\theta$ and $w_\theta$ are continuous for $\theta$.
Hence,
\begin{align}
\lim_{\theta\to 0} 
\inner{v_\theta}{w_\theta}
=
\inner{u}{w_0}
=\sum_{x}P(x)=1, \Label{28-6}
\end{align}
which implies the first equation of (\ref{28-5}).
Similarly,
\begin{align}
\lim_{\theta\to 0} 
\max_x v_\theta(x)
=
\max_x u(x)=1.\Label{28-7}
\end{align}
Combining (\ref{28-6}) and (\ref{28-7}),
we obtain the second equation of (\ref{28-5}).
\end{proof}

Using these relation, we can show the following lemma.
\begin{lemma}
For any initial distributions $P_0$ and $P_1$,
we have
\begin{align}
\Label{1-4-1}
\lim_{n \to \infty}
\frac{1}{n}D(W_0^{\times n} \times P_0 \|W_1^{\times n} \times P_1)
=&
D(W_0\|W_1), \\
\Label{1-4-2}
\lim_{n \to \infty}
\frac{1}{n}D_{1+s}(W_0^{\times n} \times P_0 \|W_1^{\times n} \times P_1)
=&
D_{1+s}(W_0\|W_1).
\end{align}
\end{lemma}

\begin{proof}
Now, we choose the functions 
$g(x,\bar{x}):= \log \frac{W_1(x|\bar{x})}{W_0(x|\bar{x})}$
and
$h(\bar{x}):= \log \frac{P_1(\bar{x})}{P_0(\bar{x})}$.
Under these choices, 
\begin{align}
\Label{1-4-3} 
D_{1+s}(W_0^{\times n} \times P_0 \|W_1^{\times n} \times P_1)
= \frac{\phi_n(-s)}{s} ,\quad
D_{1+s}(W_0\|W_1)
= \frac{\phi(-s)}{s} .
\end{align}
Hence, combining (\ref{25-12}) ad (\ref{1-4-3}), 
we obtain (\ref{1-4-2}).

Since the relative R\'{e}nyi entropy
$D_{1+s}(W_0^{\times n} \times P_0 \|W_1^{\times n} \times P_1)$
is 
monotone increasing with respect to $s$
and 
$\lim_{s\to 0} D_{1+s}(W_0^{\times n} \times P_0 \|W_1^{\times n} \times P_1)=D(W_0^{\times n} \times P_0 \|W_1^{\times n} \times P_1)$,
we have
\begin{align*}
& D_{1-\delta}(W_0\|W_1)
=
\lim_{n \to \infty}
\frac{1}{n}D_{1-\delta}(W_0^{\times n} \times P_0 \|W_1^{\times n} \times P_1) \\
\le &
\lim_{n \to \infty}
\frac{1}{n}D(W_0^{\times n} \times P_0 \|W_1^{\times n} \times P_1) \\
\le &
\lim_{n \to \infty}
\frac{1}{n}D_{1+\delta}(W_0^{\times n} \times P_0 \|W_1^{\times n} \times P_1)
=
D_{1+\delta}(W_0\|W_1)
\end{align*}
for $\delta>0$.
Since $\lim_{s\to 0}D_{1+s}(W_0 \| W_1)=D(W_0 \| W_1)$,
we obtain (\ref{1-4-1}).
\end{proof}

\section{Asymptotic variance}\Label{s13}
Firstly, we prepare the following lemma.
\begin{lemma} \label{L6}
The cumulant generating function of 
the random variable 
$\sqrt{n}(\frac{\tilde{g}^n(X^{n+1})}{n}- \eta(0))$ 
converges as follows.
\begin{align}
\log \mathsf{E}_{0} [
\exp [\delta \sqrt{n}(\frac{\tilde{g}^n(X^{n+1})}{n}- \eta(0))]] 
=
 \phi_n(\frac{\delta}{\sqrt{n}})
-{\delta}{\sqrt{n}} \eta(0)
\to \delta^2 \frac{1}{2}\frac{d^2 \phi}{d \theta^2} (0). 
\Label{28-9}
\end{align}
\end{lemma}

\begin{proof}
Using (\ref{25-12}) and (\ref{28-5}), we have
\begin{align*}
& \lim_{n \to \infty}
 \phi_n(\frac{\delta}{\sqrt{n}})
-{\delta}{\sqrt{n}} \eta(0)
\le
\lim_{n \to \infty}
n \phi(\frac{\delta}{\sqrt{n}})
-{\delta}{\sqrt{n}} \eta(0)
+ \overline{\delta}(\frac{\delta}{\sqrt{n}})
 \nonumber \\
=& 
\lim_{n \to \infty}
\delta^2 \frac{\phi\left(\frac{\delta}{\sqrt{n}}\right) - \left(\frac{\delta}{\sqrt{n}}\right) 
\frac{d \phi}{d\theta} (0)
}{\left( \frac{\delta}{\sqrt{n}} \right)^2} 
= \frac{\delta^2}{2}\frac{d^2 \phi}{d\theta^2}(0) .
\end{align*}
Similarly, the opposite inequality can be shown by (\ref{25-12}) and (\ref{28-5}).
Hence, we obtain the desired relation.
\end{proof}

The right hand side of (\ref{28-9}) is the cumulant generating function of Gaussian distribution with the variance $\frac{d^2 \phi}{d\theta^2}(0)$ and average $0$.
Since the limit of cumulant generating function uniquely decides the limit of the distribution function \cite{CRRao},
Lemma \ref{L6} reproduces the central limit theorem 
as a corollary.

\begin{corollary}\Label{Co1}\cite{CLT2,CLT3,CLT4}
The limiting distribution of 
$\sqrt{n}(\frac{\tilde{g}^n(X^{n+1}) }{n}- \eta(0))$
is characterized as
\begin{align}
\lim_{n \to \infty} 
W_0^{\times n}\times P_0\left\{ \tilde{g}^n(X^{n+1}) - n 
\eta(0) \le \sqrt{n} \delta  \right\} = 
\Phi(\frac{\delta}{\sqrt{\frac{d^2 \phi}{d\theta^2}(0)}}),
\end{align}
where
$\Phi(y):= \int_{-\infty}^{y} 
\frac{e^{-\frac{x^2}{2}}}{\sqrt{2\pi}} d x$.
\end{corollary}

The above corollary can be regarded as the Markov version of 
the central limit theorem.
The above derivation is much simpler than existing derivations \cite{CLT2,CLT3,CLT4}
because it employs only our evaluation of the cumulant generating function.
As the refinement of the above argument,
the paper \cite[Theorem 2]{Herve} showed 
the Markov version of the Berry-Essen Theorem as follows.

\begin{proposition}(\cite[Theorem 2]{Herve})
For a given constant $\delta>0$, there exists a constant $C$ such that
\begin{align}
\Label{1-3-6}
|W_0^{\times n}\times P_0\left\{ \tilde{g}^n(X^{n+1}) - n 
\eta(0) \le \sqrt{n} \delta  \right\} -
\Phi(\frac{\delta}{\sqrt{V}})|
\le \frac{C}{\sqrt{n}},
\end{align}
where $V$ is the asymptotic variance. 
\end{proposition}
For the calculation of $C$, see \cite[Theorem 2]{Herve}.
Since Corollary \ref{Co1} 
shows that the asymptotic variance is
$\frac{d^2 \phi}{d\theta^2}(0)$,
we can replace $V$ by $\frac{d^2 \phi}{d\theta^2}(0)$ in 
the above proposition.
The paper \cite[Lemma 5.3]{HW14-1} also
gives another expression of $\frac{d^2 \phi}{d\theta^2}(0)$ as follows.

\begin{proposition} \Label{L11-2}
The second derivative $\frac{d^2\phi}{d\theta^2} (0)$ is calculated as
\begin{align}
\frac{d^2\phi}{d\theta^2} (0)
=
\mathsf{V}_0 [g(X,X')]
+
2 \sum_{x,\bar{x}}
W(x|\bar{x}) g(x,\bar{x}) \frac{d \tilde{P}_{\theta}(\bar{x})}{d \theta}\Bigr|_{\theta=0},
\Label{27-12}
\end{align}
where $\mathsf{V}_0$ denotes the variance when $X,X'$ obeys the joint distribution $W_0 \times \tilde{P}_0$.
\end{proposition}

In this paper, we give another expression of 
$\frac{d^2 \phi}{d\theta^2}(0)$,
which is easier to compute in some case than the second derivative of $\phi(\theta)$ at $\theta=0$ and (\ref{27-12}).
To describe it, we define the matrices
$A_{x,\bar{x}}:= \tilde{P}_{0}(x)$, 
$W_{x,\bar{x}}:=W(x|\bar{x})$, and 
the fundamental matrix
$Z:= (I-(W-A))^{-1}$ \cite{kemeny-snell-book}, whose existence is guaranteed by the following lemma.

\begin{proposition}
\cite[Theorem 4.3.1]{kemeny-snell-book} \label{existence-of-z}
For a transition matrix $W$,
the matrix $Z = (I- (W - A))^{-1}$ exists and 
\begin{align}
Z = I + \sum_{n=1}^\infty (W^n - A) 
= \sum_{n=0}^\infty (W - A)^n.
\end{align}
We also have
$(W - A)^n = W^n - A$
for every $n$.
\end{proposition}
This proposition can be shown by the fact that $\lim_{n\to \infty}W^n = A $.
Then, we give another expression of 
$\frac{d^2 \phi}{d\theta^2}(0)$ as follows.

\begin{theorem}\Label{L20B}
\begin{align}
\nonumber
&\frac{d^2 \phi}{d \theta^2}(0)\\
=&
\mathsf{V}_0 [g(X,X')] \nonumber\\
&+
2 
\sum_{x,\bar{x}}
(\sum_{\bar{x_o}} g(x,\bar{x_o}) W(x|\bar{x_o}) )
(Z -A)_{x,\bar{x}}
(\sum_{x_o} g(\bar{x},x_o) \tilde{P}_0(x_o)W(\bar{x}|x_o) ).
\Label{1-7}
\end{align}
\end{theorem}
The proof of Theorem \ref{L20B} is given in Appendix \ref{as1}.
Combining (\ref{27-12}), we obtain 
\begin{align}
\frac{d \tilde{P}_{\theta}(\bar{x})}{d \theta}\Bigr|_{\theta=0}
=
\sum_{x}
(\sum_{\bar{x_o}} g(x,\bar{x_o}) \tilde{P}_0(\bar{x_o}) W(x|\bar{x_o}) )
(Z -A)_{x,\bar{x}}.
\Label{1-6}
\end{align}

\begin{remark}
When $g(x,\bar{x})$ is $x$ or $\bar{x}$,
the literatures \cite{CLT3,CLT4}
showed the central limit theorem 
with by using the asymptotic variance.
They did not give any expression of the asymptotic variance
without the infinite sum.
In this case, 
the paper \cite{CLT2}
showed the central limit theorem and 
the asymptotic variance equals
the second derivative of the limit $\lim_{n \to \infty}\frac{\phi_n(\theta)}{n}$.
However, it did not give a concrete form of the limit.
In this limited case,
the literatures \cite{kemeny-snell-book,Peskun} showed that the asymptotic variance 
is given as the right hand side of (\ref{1-7}),
and the paper \cite{CLT5}
gave another expression for the asymptotic variance. 
When we apply the result by \cite{kemeny-snell-book,Peskun}
to the transition matrix $P( g(X_{n+1},X_{n})=x|g(X_{n},X_{n-1})=\bar{x} )$,
we can derive a formula for the asymptotic variance
in our general case.
However, this method cannot derive as simple a formula as
our formula (\ref{1-7}).
Hence, our formula (\ref{1-7}) is useful for practical calculation.
\end{remark}

\section{Tail probability}\Label{s14}
Combining Proposition \ref{P1}, Lemma \ref{L23}, and (\ref{25-12}),
we can derive the following lower bound on the exponent
by using the function ${\phi'}^{-1}(a)$ 
$\underline{\delta}(\theta)$,
$\overline{\delta}(\theta)$,
and
$\phi(\theta)$
defined in (\ref{4-1-1}), (\ref{25-12b}), (\ref{25-12b}), 
and (\ref{4-1-2}).

\begin{theorem} \Label{P4}
For any $a > \eta(0)=\mathsf{E}_0[g]$, we have
\begin{align}
\nonumber
& - \log W_0^{\times n}\times P_0 \{ \tilde{g}^n(X^{n+1}) \ge n a \} \ge  
\sup_{\theta \ge 0}[ n \theta a - n \phi(\theta) - \overline{\delta}(\theta)] \\
=& n{\phi'}^{-1}(a) a - n\phi({\phi'}^{-1}(a)) - \overline{\delta}({\phi'}^{-1}(a))
= n D(W_{{\phi'}^{-1}(a)} \| W_{0} )- \overline{\delta}({\phi'}^{-1}(a)). 
\Label{27-32}
\end{align}
Similarly, for $a < \eta(0)=\mathsf{E}_0[g]$, we have
\begin{align*}
&- \log W_0^{\times n}\times P_0 \{ \tilde{g}^n(X^{n+1}) \le n a \} 
\ge  
\sup_{\theta \le 0}[ n \theta a - n \phi(\theta)- \overline{\delta}(\theta)]  \\
=& n {\phi'}^{-1}(a) a - n \phi({\phi'}^{-1}(a)) - \overline{\delta}({\phi'}^{-1}(a))
= n D(W_{{\phi'}^{-1}(a)} \| W_0)- \overline{\delta}({\phi'}^{-1}(a)).
\end{align*}
\end{theorem}

Combining Theorem \ref{T1} and (\ref{25-12}) of Lemma \ref{lemma:mgf-finite-evaluation},
we can derive the following converse bound.
\begin{theorem} \Label{T4}
For any $a > \eta(0)=\mathsf{E}_0[g]$, we have
\begin{align}
\nonumber 
\lefteqn{ - \log W_0^{\times n}\times P_0 \{ \tilde{g}^n(X^{n+1}) \ge n a \}  } 
 \\
\stackrel{(a)}{\le} & 
 \inf_{s > 0 \atop \theta \in \mathbb{R}, {\bar{\theta}} \le 0}
 \Bigl[ n \phi((1+s)\theta) - n (1+s)\phi(\theta) 
 + \overline{\delta}((1+s)\theta) - \underline{\delta}(\theta)
\nonumber \\ 
&- (1+s) \log \left(1- e^{ - n [{\bar{\theta}} a - \phi(\theta+{\bar{\theta}}) + \phi(\theta) 
+\overline{\delta}(\theta+{\bar{\theta}}) -\underline{\delta}(\theta)] } 
\right) 
\Bigr] / s
\nonumber\\
\stackrel{(b)}{\le} & 
 \inf_{s > 0 \atop \theta > {\phi'}^{-1}(a)} \Bigl[
n \phi((1+s)\theta) - n(1+s)\phi(\theta) 
+ \overline{\delta}((1+s)\theta) - \underline{\delta}(\theta)
\nonumber \\
& - (1+s) \log \left(1- e^{n
[(\theta - {\phi'}^{-1}(a)) a + \phi({\phi'}^{-1}(a)) - \phi(\theta)
+\overline{\delta}({\phi'}^{-1}(a)) -\underline{\delta}(\theta)
]} \right)  
\Bigr] / s 
\nonumber\\
\stackrel{(c)}{=} 
& \inf_{s > 0 \atop \theta > {\phi'}^{-1}(a)} 
n D_{1+s}(W_\theta\|W_0)
+\frac{1}{s}[\overline{\delta}((1+s)\theta) - \underline{\delta}(\theta)] \nonumber \\
 &\quad -\frac{1+s}{s}
 \log \left(1- e^{-n D(W_{{\phi'}^{-1}(a)}\|W_{\theta})
+\overline{\delta}({\phi'}^{-1}(a)) -\underline{\delta}(\theta)
} \right) .
   \Label{27-33}
\end{align}
Similarly, for any $a < \eta(0)=\mathsf{E}_0[g]$, we  have
\begin{align}
\lefteqn{ - \log W_0^{\times n}\times P_0 \{ \tilde{g}^n(X^{n+1}) \le n a \} } \nonumber \\
\le& \inf_{s > 0 \atop \theta \in \mathbb{R}, {\bar{\theta}} \ge 0}
 \Bigl[ n\phi((1+s)\theta) - n(1+s)\phi(\theta) 
+ \overline{\delta}((1+s)\theta) - \underline{\delta}(\theta)
\nonumber \\
&
 - (1+s) \log \left(1- e^{ - n [{\bar{\theta}} a - \phi(\theta+{\bar{\theta}}) + \phi(\theta) 
+\overline{\delta}(\theta+{\bar{\theta}}) -\underline{\delta}(\theta)] } \right) 
\Bigr] / s
\nonumber\\
\le& \inf_{s > 0 \atop \theta < {\phi'}^{-1}(a)} \Bigl[
 n \phi((1+s)\theta) - (n-1)(1+s)\phi(\theta) 
+ \overline{\delta}((1+s)\theta) - \underline{\delta}(\theta)
\nonumber \\
& - (1+s) \log \left(1- e^{n [(\theta - {\phi'}^{-1}(a)) a + \phi({\phi'}^{-1}(a)) - \phi(\theta)
+\overline{\delta}({\phi'}^{-1}(a)) -\underline{\delta}(\theta)
]} \right) 
\Bigr] / s 
\nonumber\\
 =& \inf_{s > 0 \atop \theta < {\phi'}^{-1}(a)} 
n D_{1+s}(W_\theta\|W_0)
+\frac{1}{s}[\overline{\delta}((1+s)\theta) - \underline{\delta}(\theta)] \nonumber \\
 &\quad -\frac{1+s}{s}
 \log \left(1- e^{-n D(W_{{\phi'}^{-1}(a)}\|W_{\theta})
+\overline{\delta}({\phi'}^{-1}(a)) -\underline{\delta}(\theta)
} \right) . \nonumber
\end{align}
\end{theorem}

\begin{proof}
$(a)$ follows from the combination of $(a)$ of Theorem \ref{T1} and 
\eqref{25-12} of Lemma \ref{lemma:mgf-finite-evaluation}.
$(b)$ and $(c)$ can be shown by the same way as $(b)$ and $(c)$ of Theorem \ref{T1}.
\end{proof}

Due to the expressions in Theorems \ref{P4} and \ref{T4},
the above upper and lower bounds are $O(1)$-computable.
These also attain the asymptotic tightness in the sense of (T2) and (T3) as follows.
From Lemma \ref{L23} and Theorems \ref{P4} and \ref{T4}, we can derive the large deviation evaluation.
\begin{corollary} \cite[Theorem 3.1.2]{DZ} \label{theorem:general-markov-ldp}
For arbitrary $\delta > 0$, we have
\begin{align*} 
\lim_{n \to \infty} - \frac{1}{n} \log W_0^{\times n}\times P_0\left\{ \tilde{g}^n (X^{n+1})- n 
\eta(0) \ge \delta n \right\} 
& = \sup_{\theta \ge 0}[{\phi'}^{-1}( 
\eta(0)+\delta) - \phi(\theta) ]
\\
\lim_{n \to \infty} - \frac{1}{n} \log W_0^{\times n}\times P_0\left\{ \tilde{g}^n (X^{n+1})- n 
\eta(0) \le - \delta n \right\}
& = \sup_{\theta \le 0}[{\phi'}^{-1}( 
\eta(0)-\delta) - \phi(\theta)].
\end{align*}
\end{corollary}


From Theorems \ref{P4} and \ref{T4}, we can derive the 
moderate deviation evaluation. 
\begin{corollary} \label{theorem:general-markov-mdp}
For arbitrary $t \in (0,1/2)$ and $\delta > 0$, we have
\begin{align} 
\lim_{n \to \infty} - \frac{1}{n^{1-2t}}  \log W_0^{\times n}\times P_0\left\{ \tilde{g}^n(X^{n+1}) - n 
\eta(0) \ge n^{1-t} \delta  \right\} 
&= \frac{\delta^2}{2 \frac{d^2\phi}{d\theta^2}(0)}
\Label{eq:general-markov-mdp-up} \\
\lim_{n \to \infty} - \frac{1}{n^{1-2t}}  \log W_0^{\times n}\times P_0\left\{ \tilde{g}^n(X^{n+1}) - n \eta(0) \le - n^{1-t} \delta  \right\} 
&= \frac{\delta^2}{2 \frac{d^2\phi}{d\theta^2}(0)}.
\end{align}
\end{corollary}
\begin{proofof}{Corollary \ref{theorem:general-markov-mdp}}
We only prove \eqref{eq:general-markov-mdp-up}. 
To show the inequality $\ge$ in \eqref{eq:general-markov-mdp-up}, we employ 
\eqref{27-32}.
That is, we substitute 
$a_n:= \eta(0)+\frac{\delta}{n^t}$
into $a$ in \eqref{eq:general-markov-mdp-up}.
Since $\frac{d {\phi'}^{-1}(\eta)}{d\eta}=
\frac{1}{\phi''({\phi'}^{-1}(\eta))}$, 
we have
${\phi'}^{-1}(a_n)= 
\frac{\delta}{\frac{d^2\phi}{d\theta^2}(0) n^t}
+o(\frac{1}{n^t})\to 0$.
Thus, Relation (\ref{28-5}) implies
$\overline{\delta}({\phi'}^{-1}(a_n))) \to 0$.
Hence, Relation (\ref{27-20}) yields that
\begin{align}
\frac{1}{n^{1-2t}}
(n D(W_{{\phi'}^{-1}(a_n)} \| W_{0} )- \overline{\delta}({\phi'}^{-1}(a_n)))
\to 
\frac{\delta^2}{2 \frac{d^2\phi}{d\theta^2}(0)}. \Label{27-30}
\end{align}
Applying (\ref{27-30}) to (\ref{27-32}), 
we obtain the part ``$\ge$'' in (\ref{eq:general-markov-mdp-up}).

To show the inequality $\le$ in \eqref{eq:general-markov-mdp-up}, we employ 
the final term of \eqref{27-33}.
That is, we substitute 
$a_n:= \eta(0)+\frac{\delta}{n^t}$
and $\theta_n:= {\phi'}^{-1}(a_n)+ \frac{\xi}{n^t} 
\frac{d^2\phi}{d\theta^2}(0)^{-1}$
into $a$ and $\theta$ in the final term of \eqref{27-33}.
Then, we have 
$\theta_n= \frac{\delta+\xi}{\frac{d^2\phi}{d\theta^2}(0)n^t}
+o(\frac{1}{n^t})\to 0$.
Thus, Relation (\ref{28-5}) implies that
$\frac{1}{s}[\overline{\delta}((1+s)\theta_n) - \underline{\delta}(\theta_n)]\to 0$
and
$\overline{\delta}({\phi'}^{-1}(a_n)) -\underline{\delta}(\theta_n)
\to 0$.
We also have $n D(W_{{\phi'}^{-1}(a_n)}\|W_{\theta_n}) \to \infty$. 
Hence, (\ref{27-21}) yields that
\begin{align}
& \lim_{n \to \infty}
 \frac{1}{n^{1-2t}}
\Bigg[ \inf_{s > 0 \atop \theta > {\phi'}^{-1}(a_n)} 
n D_{1+s}(W_\theta\|W_0)
+\frac{1}{s}[\overline{\delta}((1+s)\theta) - \underline{\delta}(\theta)] 
\nonumber \\
& -\frac{1+s}{s}
 \log \left(1- e^{-n D(W_{{\phi'}^{-1}(a_n)}\|W_{\theta})
+\overline{\delta}({\phi'}^{-1}(a_n)) -\underline{\delta}(\theta)
} \right)\Bigg] \nonumber \\
\le &
 \lim_{n \to \infty}
 \frac{1}{n^{1-2t}}
\Bigg[n D_{1+s}(W_{\theta_n}\|W_0)
+\frac{1}{s}[\overline{\delta}((1+s)\theta_n) - \underline{\delta}(\theta_n)]
\nonumber \\
& -\frac{1+s}{s}
 \log \left(1- e^{-n D(W_{{\phi'}^{-1}(a_n)}\|W_{\theta_n})
+\overline{\delta}({\phi'}^{-1}(a_n)) -\underline{\delta}(\theta_n)
} \right) \Bigg] \nonumber \\
=&
\lim_{n \to \infty}
\frac{n}{n^{1-2t}} D_{1+s}(W_{\theta_n}\|W_0)
=
\frac{(\delta+\xi)^2}{2 \frac{d^2\phi}{d\theta^2}(0)}
(1+s).
\Label{27-31}
\end{align}
Finally, we take the limits $\xi\to 0$ and $s\to 0$.
Then,
applying (\ref{27-31})
to (\ref{27-33}), 
we obtain the part ``$\ge$'' in (\ref{eq:general-markov-mdp-up}).
\end{proofof}

\section{Simple hypothesis testing}\Label{s15}
Next, we consider the binary simple hypothesis testing.
To formulate the binary simple hypothesis testing, we consider the case 
that the null and alternative hypotheses are $P_0$ and $P_1$.
For theoretical simplicity, we often focus on randomized tests as follows.
We fix a random variable $T$ taking values in the interval $[0,1]$, which is called a test.
When we observe $T=t$, we support the alternative hypothesis with probability $t$
and support the null hypothesis with probability $1-t$.
Then, the first and the second error probabilities are given as
$\mathsf{E}_{P_1}[T] $ and $\mathsf{E}_{P_0}[1-T]$,
where $\mathsf{E}_{P_i}$ denotes the expectation under the distribution $P_i$.
When we choose the random variable $T$ to be the test function with support $S$, 
the random variable $T$ realizes the test whose rejection region $S$.

Then, we consider the following value.
\begin{align}
\beta_\epsilon(P_1 \| P_0):=
\min_{T}\{ \mathsf{E}_{P_0}[1-T] | \mathsf{E}_{P_1}[T] \le \epsilon\}
=
\min_{T}\{ \mathsf{E}_{P_0}[1-T] | \mathsf{E}_{P_1}[T] = \epsilon\}.
\end{align}
Since we allow randomized tests,
the optimum test $T$ is realized with the condition $\mathsf{E}_{P_1}[T] = \epsilon$.

Now, we consider the hypothesis testing with the two hypotheses
$W_0^{n-1}\times P_0$ and $W_1^{n-1}\times P_1$.
Then, we choose the functions 
$g(x,\bar{x}):= \log \frac{W_1(x|\bar{x})}{W_0(x|\bar{x})}$
and
$h(\bar{x}):= \log \frac{P_1(\bar{x})}{P_0(\bar{x})}$.
Under these choices, 
$\phi(1)=0$ and
we can evaluate the minimum 2nd error probability
by using the functions 
$\hat{\theta}(r)$,
$\underline{\delta}(\theta)$,
$\overline{\delta}(\theta)$,
and
$\phi(\theta)$
defined in (\ref{12-31-1}), (\ref{25-12b}), (\ref{25-12b}), 
and (\ref{4-1-2})
as well as the relative entropies
$D(W_{\hat{\theta}(r)}\|W_{\theta})$
and $D_{1+s}(W_{\theta}\|W_0)$.

\begin{theorem}\Label{T27}
The minimum 2nd error probability  
$\beta_{e^{-nr}}(W_1^{\times n}\times P_1 \| W_0^{\times n}\times P_0) $
defined in (\ref{23-1}) satisfies 
\begin{align}
& 
 \sup_{0 \le \theta \le 1} \frac{n (-\theta  r - \phi(\theta))-
\underline{\delta}(\theta) }{1-\theta}
\nonumber \\
\le & -\log \beta_{e^{-nr}}(W_1^{\times n}\times P_1 \| W_0^{\times n}\times P_0) \nonumber \\
\stackrel{(a)}{\le} & \inf_{\bar{\theta}\ge 0, s>0,\theta \in (0,1)}
\frac{1}{s} [n (\phi( (1+s) \theta) -(1+s) \phi(\theta)) 
+( \overline{\delta}( (1+s) \theta) -(1+s) \underline{\delta}(\theta))\nonumber \\
&\quad -(1+s) \log (1-
2 e^{
n\frac{ -(1+\bar{\theta})\phi(\theta)+\phi( (1+\bar{\theta})\theta-\bar{\theta} ) -\bar{\theta} r}{1+\bar{\theta}}
+\frac{ -(1+\bar{\theta})\underline{\delta}(\theta)+\overline{\delta}( (1+\bar{\theta})\theta-\bar{\theta} )}{1+\bar{\theta}}
}
) 
]\nonumber\\
\stackrel{(b)}{\le}
 & \inf_{s>0,\theta \in (\hat{\theta}(r),1)}
\frac{1}{s} [n (\phi( (1+s) \theta) -(1+s) \phi(\theta))
+( \overline{\delta}( (1+s) \theta) -(1+s) \underline{\delta}(\theta))\nonumber \\
&\quad 
-(1+s) \log (1-
2 e^{n(-\phi(\theta)+\phi (\hat{\theta}(r))
+(\theta-\hat{\theta}(r)) \frac{d\phi}{d\theta}(\hat{\theta}(r)) )
-\underline{\delta}(\theta)
+\frac{(1-\theta) \overline{\delta}(\hat{\theta}(r)) }{1-\hat{\theta}(r)} 
}
) ]\nonumber \\
\stackrel{(c)}{=} & \inf_{s>0,\theta \in (\hat{\theta}(r),1)}
n D_{1+s}(W_{\theta}\|W_0) 
+\frac{1}{s}( \overline{\delta}( (1+s) \theta) -(1+s) \underline{\delta}(\theta))\nonumber \\
&\quad- \frac{1+s}{s} \log (1- 2 e^{-n D(W_{\hat{\theta}(r)}\|W_{\theta})
-\underline{\delta}(\theta)
+\frac{(1-\theta) \overline{\delta}(\hat{\theta}(r)) }{1-\hat{\theta}(r)} }). 
\nonumber
\end{align}
\end{theorem}

\begin{proof}
The inequality $(a)$ can be shown by combining (\ref{26-6}) and (\ref{25-12}).
To show $(b)$ and $(c)$, 
we restrict $\theta$ in $[\hat{\theta}(r),1]$
and choose $\bar{\theta}$ to be 
$\frac{\theta-\hat{\theta}(r)}{1-\theta} \ge 0$ similar to the proof of (\ref{27-3}).
Then, 
\begin{align*}
&\frac{ -(1+\bar{\theta})\underline{\delta}(\theta)+\overline{\delta}( (1+\bar{\theta})\theta-\bar{\theta} ) }{1+\bar{\theta}} 
=-\underline{\delta}(\theta)
+\frac{\overline{\delta}(\hat{\theta}(r)) }{1+\bar{\theta}} 
=-\underline{\delta}(\theta)
+\frac{(1-\theta) \overline{\delta}(\hat{\theta}(r)) }{1-\hat{\theta}(r)} .
\end{align*}
As is shown in the proof of (\ref{27-3}),
we have
\begin{align*}
& \frac{ -(1+\bar{\theta})\phi(\theta)+\phi( (1+\bar{\theta})\theta-\bar{\theta} ) -\bar{\theta} r}{1+\bar{\theta}}\nonumber \\
=&
-\phi(\theta)+\phi(\hat{\theta}(r))
+(\theta-\hat{\theta}(r)) \frac{d\phi}{d\theta}(\hat{\theta}(r)) 
= D(W_{\hat{\theta}(r)}\|W_{\theta}).
\end{align*}
Hence, we obtain $(b)$ and $(c)$.
\end{proof}

Due to the expressions in Theorem \ref{T27},
the above upper and lower bounds are $O(1)$-computable.
These also attain the asymptotic tightness in the sense of (H2) and (H3) as follows.
From Lemma \ref{L1-3-1} and Theorem \ref{T27},
we can recover the Hoeffding type evaluation as follows.
\begin{corollary}\cite[Theorem 2]{N}\cite[Theorem 1]{NK}
\begin{align}
\nonumber
&\lim_{n \to \infty}-\frac{1}{n}\log \beta_{e^{-nr}}(W_1^{\times n}\times P_1 \| W_0^{\times n}\times P_0) 
\\
=&  
\sup_{0 \le \theta \le 1} \frac{-\theta  r - \phi(\theta)}{1-\theta} 
= \sup_{0 \le \theta \le 1} 
\frac{\theta ( -r + D_{1-\theta}(W_0\|W_1))}{1-\theta} \nonumber \\
=&\inf_{s>0,\theta \in (0,\hat{\theta}(r))}
D_{1+s}(W_{\theta}\|W_0) 
= D(W_{\hat{\theta}(r)}\|W_0)\nonumber\\
=& \min_{W:D(W\|W_1) \le r} D(W\|W_0).
\Label{5-7}
\end{align}
\end{corollary}

\begin{remark}
Natarajan \cite[Theorem 2]{N} 
showed that the exponent (\ref{5-7}) equals $\min_{W:D(W\|W_1) \le r} D(W\|W_0)$.
Nakagawa and Kanaya \cite[Theorem 1]{NK}
showed that the exponent (\ref{5-7}) equals $D(W_{\hat{\theta}(r)}\|W_0)$.
They did not consider other expressions in (\ref{5-7}).
\end{remark}

From Theorem \ref{T27},
we obtain the following moderate deviation type evaluation.
\begin{corollary}\Label{C1-3-1}
For $t \in (0,\frac{1}{2})$, we have
\begin{align}
\Label{1-3-1}
&\lim_{n \to \infty}-\frac{1}{n^{1-2t}}\log \beta_{e^{-n D(W_0\|W_1) + n^{1-t} \delta }}
(W_1^{\times n}\times P_1 \| W_0^{\times n}\times P_0) 
= 
\frac{\delta^2}{2 \frac{d^2\phi}{d \theta^2}(0) }.
\end{align}
That is,
\begin{align}
\Label{1-3-2}
-\log \beta_{e^{- n^{1-2t}r}}( W_0^{\times n}\times P_0 \|W_1^{\times n}\times P_1 ) 
= 
D(W_0\|W_1) n -  \sqrt{2 \frac{d^2\phi}{d \theta^2}(0) r}n^{1-t}+o(n^{1-t}).
\end{align}
\end{corollary}

\begin{proofof}{Corollary \ref{C1-3-1}}
First, we show (\ref{1-3-1}) in the same way as the proof of (\ref{eq:general-markov-mdp-up}).
(\ref{26-1}) implies
\begin{align*}
& \sup_{0 \le \theta \le 1} 
\frac{n (-\theta  r - \phi(\theta))-
\underline{\delta}(\theta) }{1-\theta}
\ge
\frac{n (-{\phi'}^{-1}(-r)  r - \phi({\phi'}^{-1}(-r)))-
\underline{\delta}({\phi'}^{-1}(-r)) }{1-{\phi'}^{-1}(-r)}\\
=&\frac{n D(W_{{\phi'}^{-1}(-r)} \| W_{0} )- \underline{\delta}({\phi'}^{-1}(-r))}{1-{\phi'}^{-1}(-r)}.
\end{align*}
Now, we choose $r_n:=D(W_0\|W_1) + \delta n^{-t}$.
Then, we have ${\phi'}^{-1}(-r_n)
= \frac{\delta}{2 \frac{d^2\phi}{d\theta^2}(0) n^t}
+o(\frac{1}{n^t})\to 0$.
Thus,
\begin{align}
\frac{n D(W_{{\phi'}^{-1}(-r_n)} \| W_{0} )- \underline{\delta}({\phi'}^{-1}(-r_n))}{1-{\phi'}^{-1}(-r_n)}
\to 
\frac{\delta^2}{2 \frac{d^2\phi}{d\theta^2}(0)}. 
\Label{1-3-3}
\end{align}
Applying (\ref{1-3-3}) to Theorem \ref{T27},
we obtain the part ``$\ge$'' in (\ref{1-3-1}).
 
Next, we choose 
$\theta_n:= {\phi'}^{-1}(a_n)+ \frac{\xi}{n^t} \frac{d^2\phi}{d\theta^2}(0)^{-1}$.
Then, applying the right hand side of $(c)$ of Theorem \ref{T27},
we obtain the part ``$\le$'' in (\ref{1-3-1})
as the same way as the proof of the part ``$\le$'' in (\ref{eq:general-markov-mdp-up}).

Solving $\frac{\delta^2}{2 \frac{d^2\phi}{d \theta^2}(0) }=r
\delta^2$, we have $\delta
=\sqrt{2 \frac{d^2\phi}{d \theta^2}(0) r}$.
Hence, we have (\ref{1-3-2}) from (\ref{1-3-1}).
\end{proofof}

We also have another type evaluation for the second kind of error probability.
\begin{lemma}\Label{L31c}
When we choose $g(x,\bar{x})=\log \frac{W_1(x|\bar{x})}{W_0(x|\bar{x})}$ and $\hat{g}(x)=\log \frac{P_1(x)}{P_0(x)}$, 
we have
\begin{align}
\nonumber
& \sup_a \{ a | W_1^{\times n}\times P_1\{ g^n(X^{n+1}) < a \} \le \epsilon \}  \\
\le &-\log \beta_\epsilon(W_1^{\times n}\times P_1 \| W_0^{\times n}\times P_0)  \nonumber \\
\le & 
\inf_{\delta>0, a} 
\left\{ a -\log \delta \left| 
W_1^{\times n}\times P_1\{ g^n(X^{n+1}) < a \} \ge \epsilon +\delta
\right.\right\}.
\Label{1-3-7}
\end{align}
\end{lemma}
Lemma \ref{L31c} can be shown by substituting 
$W_i^{\times n}\times P_i$ into $P_i$ ($i=0,1$) 
in Lemma \ref{L31} in Appendix \ref{s3}.

Combining (\ref{1-3-6}) and (\ref{1-3-7}),
we can derive lower and upper $O(1)$-computable bounds
for $\beta_\epsilon(W_1^{\times n}\times P_1 \| W_0^{\times n}\times P_0)$.
Applying Corollary \ref{Co1} to the random variable 
$\log \frac{W_1^n \times P_1(X^{n+1})}{W_0^n \times P_0(X^{n+1})} $
in Lemma \ref{L31c}, we obtain the Stein-Strassen type evaluation.
That is, these bounds attain the asymptotic tightness in the sense of (H1).

\begin{theorem}
\begin{align*}
-\log \beta_{\epsilon}(W_1^{\times n}\times P_1 \| W_0^{\times n}\times P_0) 
= n D(W_1\|W_0)+ \sqrt{n} \sqrt{\frac{d^2 \phi}{d\theta^2}(1)} \Phi^{-1}(\epsilon)+o(\sqrt{n}).
\end{align*}
\end{theorem}

\section{Conclusion}\Label{s9}
We have derived upper and lower 
$O(1)$-computable
bounds of the cumulant generating function 
of the Markov chain by using the convex function $\phi(\theta)$.
Using these bounds, we have given an simple alternative proof of the
central limit theorem of the sample mean in the Markovian chain.
Also, using these bounds, we have derived upper and lower $O(1)$-computable bounds of the tail probability of the sample mean,
which attains the asymptotic tightness in the sense of (T2) and (T3).
Using the above upper and lower bounds,
we have derived upper and lower $O(1)$-computable
bounds of the minimum error probability of the 2nd kind of error 
under the constraint for the error probability of the 1st kind of error,
which attains the asymptotic tightness in the sense of (H2) and (H3).
These bounds have not been derived even in the independently and identically distributed case.
We have also derived other upper and lower $O(1)$-computable
bounds 
that attains the asymptotic tightness in the sense of (H1).

However, in this paper, we have assumed that 
our system consists of finite elements.
Indeed, 
the existing papers \cite{C1,C2,C3} reported 
several difficulties to evaluate the tail probability
of the sample mean in the continuous probability space 
even with the discrete time Markov chain.
So, it is remained as a challenging problem 
to extend the obtained results to the continuous case. 
This extension will enable us to handle several 
Gaussian Markovian chains in a simple way. 
Further, the obtained bounds are useful for several topics in information theory \cite{W-H2}.

\section*{Acknowledgment}
MH is partially supported by a MEXT Grant-in-Aid for Scientific Research (A) No. 23246071. 
MH is also partially supported by the National Institute of Information and Communication Technology (NICT), Japan.
SW is partially supported by JSPS Postdoctoral Fellowships for Research Abroad.
The Centre for Quantum Technologies is funded by the Singapore Ministry of Education and the National Research Foundation as part of the Research Centres of Excellence programme.

\appendix

\section{Exponential family of distributions}\Label{s3}
In this appendix, we discuss several formulas in an exponential family of distributions 
$\{P_\theta\}$
with single observation
when $P_\theta(x):= P(x) e^{\theta x-\phi(\theta)}$
with cumulant generating function $\phi(\theta):=\log \sum_x P(x) [e^{\theta x}]$.

The exponential families of transition matrices
contain
exponential families of distributions 
by considering the family of transition matrices 
$W_{\theta}(x|\bar{x}):=P_\theta(x)$ from the family of distributions $P_\theta$.
Hence, the definitions and the notations given in Section \ref{s4} 
are applied to 
the exponential family of distributions 
$\{P_\theta\}$ in the following.

\subsection{Tail probability}
First, we define the relative entropy and the relative R\'{e}nyi entropy between two distributions $P$ and $\bar{P}$
are given as
\begin{align}
D(P\|\bar{P}) &:=
\sum_x P(x) \log \frac{P(x)}{\bar{P}(x)} \Label{28-1} \\
D_{1+s}(P\|\bar{P}) &:=
\frac{1}{s}\log \sum_x P(x)^{1+s} \bar{P}(x)^{-s} .
\Label{28-2}
\end{align}
Using the cumulant generating function $\phi(\theta)$,
we investigate the lower bound on the tail probability as follows.
The following lower bound on the tail probability is nothing but 
Cram\'er's theorem in the large deviation theory \cite{DZ}. 
\begin{proposition} \Label{P1}
For any $a > \mathsf{E}[X]$, we have
\begin{eqnarray*}
- \log P_0\{ X \ge a \} \ge \sup_{\theta \ge 0}[ \theta a - \phi(\theta) ] 
= {\phi'}^{-1}(a) a - \phi({\phi'}^{-1}(a)) 
= D(P_{{\phi'}^{-1}(a)} \| P_0 ).
\end{eqnarray*}
Similarly, for $a < \mathsf{E}[X]$, we have
\begin{eqnarray*}
- \log P_0\{ X \le a \} \ge \sup_{\theta \le 0}[ \theta a - \phi(\theta)] 
= {\phi'}^{-1}(a) a - \phi({\phi'}^{-1}(a)) 
= D(P_{{\phi'}^{-1}(a)} \| P_0).
\end{eqnarray*}
\end{proposition}

By using the monotonicity of the R\'enyi relative entropy \cite{Csiszar}, we can derive the following converse bound.
\begin{theorem} \Label{T1}
For any $a > \mathsf{E}[X]$, we have
\begin{align}
\lefteqn{ - \log P\{ X \ge a \} } \nonumber \\
\stackrel{(a)}{\le} & \inf_{s > 0 \atop \theta \in \mathbb{R}, {\bar{\theta}} \le 0}
 \left[ \phi((1+s)\theta) - (1+s)\phi(\theta) 
 - (1+s) \log \left(1- e^{ - [{\bar{\theta}} a - \phi(\theta+{\bar{\theta}}) + \phi(\theta) ] } \right) \right] / s
\nonumber \\
\stackrel{(b)}{\le} & \inf_{s > 0 \atop \theta > {\phi'}^{-1}(a)} \left[
 \phi((1+s)\theta) - (1+s)\phi(\theta) 
 - (1+s) \log \left(1- e^{(\theta - {\phi'}^{-1}(a)) a + \phi({\phi'}^{-1}(a)) - \phi(\theta)} \right) \right] / s 
\nonumber \\
\stackrel{(c)}{=} & \inf_{s > 0 \atop \theta > {\phi'}^{-1}(a)} 
D_{1+s}(P_\theta\|P_0)
-\frac{1+s}{s} \log \left(1- e^{-D(P_{{\phi'}^{-1}(a)}\|P_{\theta})} \right) . 
  \nonumber
\end{align}
Similarly, for any $a < \mathsf{E}[X]$, we  have
\begin{align}
\lefteqn{ - \log P\{ X \le a \} } \nonumber \\
\stackrel{(d)}{\le}
& \inf_{s > 0 \atop \theta \in \mathbb{R}, {\bar{\theta}} \ge 0}
 \left[ \phi((1+s)\theta) - (1+s)\phi(\theta) 
 - (1+s) \log \left(1- e^{ - [{\bar{\theta}} a - \phi(\theta+{\bar{\theta}}) + \phi(\theta) ] } \right) \right] / s
\nonumber \\
\stackrel{(e)}{\le}
& \inf_{s > 0 \atop \theta < {\phi'}^{-1}(a)} \left[
 \phi((1+s)\theta) - (1+s)\phi(\theta) 
 - (1+s) \log \left(1- e^{(\theta - {\phi'}^{-1}(a)) a + \phi({\phi'}^{-1}(a)) - \phi(\theta)} \right) \right] / s 
\nonumber
\\
\stackrel{(f)}{=}
& \inf_{s > 0 \atop \theta < {\phi'}^{-1}(a)} 
D_{1+s}(P_\theta\|P_0)
-\frac{1+s}{s} \log \left(1- e^{-D(P_{{\phi'}^{-1}(a)}\|P_{\theta})} \right) . 
\nonumber
\end{align}
\end{theorem}
\begin{proof}
We only show $(a)$-$(c)$.
We can show $(d)$-$(f)$
almost in a similar manner.
For arbitrary $\theta \in \mathbb{R}$, 
we set $\alpha := P\{ X \ge a\}$ and $\beta := P_{\theta}\{ X \ge a \}$.
Then, by the monotonicity of the R\'enyi relative entropy 
\cite{Csiszar}, we have
\begin{eqnarray*}
D_{1+s}(P_{\theta}\|P) \ge \frac{1}{s}\log\left[ \beta^{1+s} \alpha^{-s} + (1-\beta)^{1+s} (1-\alpha)^{-s}\right] 
\ge \frac{1}{s} \log \beta^{1+s} \alpha^{-s}.
\end{eqnarray*}
Thus, we have
\begin{eqnarray*}
- \log \alpha \le \frac{\phi((1+s)\theta) - (1+s)\phi(\theta) - (1+s)\log \beta}{s}.
\end{eqnarray*}

Now, for any ${\bar{\theta}} \le 0$, we have
\begin{align}
& 1 - \beta
= P_{\theta}\{ X < a \} 
\le \sum_x P_{\theta}(x) e^{{\bar{\theta}}(x-a)} \nonumber \\
=& \sum_x P(x) e^{(\theta+{\bar{\theta}})x - {\bar{\theta}} a - \phi(\theta)} 
= e^{- [{\bar{\theta}} a - \phi(\theta+{\bar{\theta}}) + \phi(\theta)] }.\nonumber
\end{align}
Thus,  $- \log \alpha \le f(s,\theta,\bar{\theta})$, 
where $f(s,\theta,\bar{\theta})$ is the function inside of the RHS of $(a)$.
Hence, we have $(a)$. 

Restricting the range of $\theta$ as $\theta > {\phi'}^{-1}(a)$,
we have
$\inf_{s > 0, \theta \in \mathbb{R}, {\bar{\theta}} \ge 0}
f(s,\theta,\bar{\theta})
\le
\inf_{s > 0, \theta > {\phi'}^{-1}(a), {\bar{\theta}} \ge 0}
f(s,\theta,\bar{\theta})$.
This restriction yields 
\begin{eqnarray*}
\sup_{{\bar{\theta}} \le 0}  [{\bar{\theta}} a - \phi(\theta+{\bar{\theta}}) + \phi(\theta)] 
 = ({\phi'}^{-1}(a)-\theta) a - \phi({\phi'}^{-1}(a)) + \phi(\theta),
\end{eqnarray*}
which is achieved by ${\bar{\theta}} = {\phi'}^{-1}(a) - \theta$.
Thus, 
since $\inf_{s > 0 , \theta > {\phi'}^{-1}(a), {\bar{\theta}} \ge 0}
f(s,\theta,\bar{\theta})$
equals the RHS of $(b)$, we have $(b)$.
Furthermore, $(c)$
can be obtained from the relations \eqref{28-1} and \eqref{28-2}.
\end{proof}

\subsection{Simple hypothesis testing}
For simple hypothesis testing, we have the following lemma for the null and alternative hypotheses are $P_0$ and $P_1$.
In fact, when two distributions $P$ and $Q$ are given on the probability space ${\cal X}$, the one-parametric exponential family
$P_\theta$ generated by the random variable $Y:=\log \frac{Q(X)}{P(X)}$ 
satisfies that $ P_0=P$ and $P_1=Q$.
Hence, the above case covers the most general setting for the binary hypothesis testing.

\begin{lemma}\Label{L31}
\begin{align*}
& \sup_a \{ a | P_1\{ \log \frac{P_1(x)}{P_0(x)} < a \} \le \epsilon \} 
\le -\log \beta_\epsilon(P_1 \| P_0) \nonumber \\
\le & 
\inf_{\delta>0, a} \left\{ a -\log \delta \left| 
P_1\{ \log \frac{P_1(x)}{P_0(x)} < a \} \ge \epsilon +\delta
\right.\right\}.
\end{align*}
\end{lemma}

\begin{proof}
Let $S_a$ be the set 
$\{ \log \frac{P_1(x)}{P_0(x)} < a \}
=\{ P_1(x) < e^a P_0(x) \}$
and $T_a$ be the test function with the support $S_a$.
When $\mathsf{E}_{P_1}[T_a] \le \epsilon$, 
\begin{align}
& e^{-a} \ge e^{-a}P_1\{ \log \frac{P_1(x)}{P_0(x)} \ge a \}
\ge P_0\{ \log \frac{P_1(x)}{P_0(x)} \ge a \} \\
=& \mathsf{E}_{P_0}[1-T_a] \ge \beta_\epsilon(P_1 \| P_0).
\end{align}
Taking the logarithm, we have
\begin{align}
a \le -\log \beta_\epsilon(P_1 \| P_0) .
\end{align}
Taking the supremum for $a$, we obtain the first inequality. 

Assume that $P_1\{ \log \frac{P_1(x)}{P_0(x)} < a \} \ge \epsilon +\delta$.
We have
\begin{align}
\epsilon + e^a \mathsf{E}_{P_0}[1-T]
=
\mathsf{E}_{P_1} [T] + e^a \mathsf{E}_{P_0}[1-T]
\ge
\mathsf{E}_{P_1} [T_a] + e^a \mathsf{E}_{P_0}[1-T_a]
\ge
\epsilon +\delta 
\end{align}
Thus,
\begin{align}
 \mathsf{E}_{P_0}[1-T]
\ge e^{-a}\delta .
\end{align}
Taking the minimum for $T$, we have
$\beta_\epsilon(P_1 \| P_0) \ge e^{-a}\delta$, which implies that
\begin{align}
-\log \beta_\epsilon(P_1 \| P_0) \le  a -\log \delta .
\end{align}
Taking the infimum for $a, \delta>0$, we obtain the second inequality.
\end{proof}

Here, we employ $\hat{\theta}(r)=\hat{\theta}[\phi](r)$
defined at \eqref{12-31-1} for a convex function $\phi$. 
Then, modifying Proposition \ref{P1},
we have the following lemma.

\begin{lemma}
We have
\begin{align*}
- \log P_1\{ Y \le \eta(\hat{\theta}(r)) \}
&\ge D(P_{\hat{\theta}(r)}\|P_{1})= r ,\nonumber \\
- \log P_0\{ Y \ge \eta(\hat{\theta}(r)) \}
&\ge D(P_{\hat{\theta}(r)}\|P_{0}).
\end{align*}
Choosing the rejection region $\{ Y \le \eta(\hat{\theta}(r)) \}$, we have
\begin{align}
-\log \beta_{e^{-r}}(P_1 \| P_0) 
\ge
\sup_{0 \le \theta \le 1} \frac{-\theta r - \phi(\theta) }{1-\theta}.
\end{align}
\end{lemma}

As the opposite inequality, we have the following lemma.
\begin{lemma}
\begin{align}
 & -\log \beta_{e^{-r}}(P_1 \| P_0) \nonumber \\
\Label{26-6}\le & \inf_{\bar{\theta}\ge 0, s>0,\theta \in (0,1)}
\frac{1}{s} [\phi( (1+s) \theta) -(1+s) \phi(\theta)  \\
&\quad -(1+s) \log (1-
2 e^{
\frac{ -(1+\bar{\theta})\phi(\theta)+\phi( (1+\bar{\theta})\theta-\bar{\theta} ) -\bar{\theta} r}{1+\bar{\theta}}}
) ] \nonumber \\
\Label{27-3} \le & \inf_{s>0,\theta \in (\hat{\theta}(r),1)}
\frac{1}{s} [\phi( (1+s) \theta) -(1+s) \phi(\theta) \\
&\quad -(1+s) \log (1-
2 e^{-\phi(\theta)+\phi(\hat{\theta}(r))
+(\theta-\hat{\theta}(r)) \frac{d\phi}{d\theta}(\hat{\theta}(r)) }
) ] \nonumber\\
= & \inf_{s>0,\theta \in (\hat{\theta}(r),1)}
D_{1+s}(P_{\theta}\|P_0) - \frac{1+s}{s} \log (1- 2 e^{-D(P_{\hat{\theta}(r)}\|P_{\theta})}) .
\Label{27-2} 
\end{align}
\end{lemma}

\begin{proof}
We choose the rejection region $S$ as $P_1(S)\le e^{-r}$.
The monotonicity of relative R\'{e}nyi entropy \cite{Csiszar} implies that
\begin{align*}
& D_{1+s}(P_\theta\|P_0) \ge 
\frac{1}{s} \log [ 
P_\theta (S)^{1+s} P_0 (S)^{-s}+(1-P_\theta (S))^{1+s} (1-P_0 (S))^{-s}] \nonumber \\
\ge &
\frac{1}{s} \log [ 
(1-P_\theta (S))^{1+s} (1-P_0 (S))^{-s}]
=
- \log (1-P_0 (S))+\frac{1+s}{s}\log (1-P_\theta (S))
\end{align*}
for $s >0$.
Hence, we have
\begin{align}
\nonumber
&-\log (1-P_0 (S)) \le
D_{1+s}(P_\theta\| P_0) -\frac{1+s}{s}\log (1-P_\theta (S)) \\
=&
\frac{1}{s} [\phi( (1+s) \theta) -(1+s) \phi(\theta)
-(1+s) \log (1-P_\theta (S)) ].
\Label{26-7}
\end{align}
Next, we focus on the inequality
\begin{align*}
(1-P_\theta (S))
+ e^{\gamma} P_{1} (S)
\ge
P_{\theta}\{\log \frac{P_{1}(x)}{P_{\theta}(x)} \ge -\gamma \}
+ e^{\gamma} P_{1} \{\log \frac{P_{1}(x)}{P_{\theta}(x)} < -\gamma \},
\end{align*}
which implies that
\begin{align*}
(1-P_\theta (S))
+ e^{\gamma} P_{1} (S)
\ge
P_{\theta}\{\log \frac{P_{1}(x)}{P_{\theta}(x)} \ge - \gamma \}.
\end{align*}
Hence,
\begin{align*}
P_{\theta}\{\log \frac{P_{1}(x)}{P_{\theta}(x)} < -\gamma \}+ e^{\gamma-r} 
\ge
P_\theta (S).
\end{align*}
For any $\bar{\theta}\ge 0$, we have
\begin{align*}
P_{\theta}\{\log \frac{P_{1}(x)}{P_{\theta}(x)} < -\gamma \}
\le \sum_{x} P_{\theta}(x)^{1+\bar{\theta}} P_{1}(x)^{-\bar{\theta}} e^{-\bar{\theta} \gamma}
= e^{ -(1+\bar{\theta})\phi(\theta)+\phi( (1+\bar{\theta})\theta-\bar{\theta} ) -\bar{\theta} \gamma}.
\end{align*}
Note that $\phi(1) =0$.
Choosing $\gamma$ so that 
$ -(1+\bar{\theta})\phi(\theta)+\phi( (1+\bar{\theta})\theta-\bar{\theta} )  -\bar{\theta} \gamma=
\gamma-r$, i.e., 
$\gamma= 
\frac{ -(1+\bar{\theta})\phi(\theta)+\phi( (1+\bar{\theta})\theta-\bar{\theta} ) +r}{1+\bar{\theta}}$, we have
\begin{align}
P_\theta (S)
\le 2 e^{
\frac{ -(1+\bar{\theta})\phi(\theta)+\phi( (1+\bar{\theta})\theta-\bar{\theta} ) -\bar{\theta} r}{1+\bar{\theta}}}.\Label{26-8}
\end{align}
Combining (\ref{26-7}) and (\ref{26-8}), we obtain (\ref{26-6}).

In the following, 
we restrict $\theta$ in $[\hat{\theta}(r),1]$.
Then, we can choose $\bar{\theta}$ to be 
$\frac{\theta-\hat{\theta}(r)}{1-\theta} \ge 0$.
Thus, using (\ref{12-31-1}) i.e., the relation 
$(\hat{\theta}(r)-1)\frac{d \phi}{d\theta}(\hat{\theta}(r))
-\phi(\hat{\theta}(r))=r$,
we have
\begin{align}
& \frac{ -(1+\bar{\theta})\phi(\theta)+\phi( (1+\bar{\theta})\theta-\bar{\theta} ) -\bar{\theta} r}{1+\bar{\theta}}
=
-\phi(\theta)+
\frac{\phi(\hat{\theta}(r)) -\bar{\theta} r}{1+\bar{\theta}} \nonumber \\
=&
-\phi(\theta)+
\frac{(1-\theta)\phi(\hat{\theta}(r)) -(\theta-\hat{\theta}(r)) r}{1-\hat{\theta}(r)} \nonumber \\
=&
-\phi(\theta)+
\frac{(1-\theta)\phi(\hat{\theta}(r)) -(\theta-\hat{\theta}(r)) 
((\hat{\theta}(r)-1) \frac{d\phi}{d\theta}(\hat{\theta}(r))
-\phi(\hat{\theta}(r)))}{1-\hat{\theta}(r)} \nonumber \\
=&
-\phi(\theta)+\phi(\hat{\theta}(r))
+(\theta-\hat{\theta}(r)) \frac{d\phi}{d\theta}(\hat{\theta}(r)) 
= D(P_{\hat{\theta}(r)}\|P_{\theta}).\nonumber
\end{align}
Hence, we obtain (\ref{27-3}) and (\ref{27-2}).
\end{proof}

\section{Proof Theorem \protect{\ref{L20B}}}\Label{as1}
Firstly, we prepare the following lemma and corollary, which will be used later.
\begin{lemma}[Ces\'aro Summability] \label{lemma:cesaro}
Suppose that a sequence of matrices $\{ \beta_n \}_{n=1}^\infty$ satisfies $\beta_n \to \beta$.
Then, we have
$\lim_{n \to \infty} \frac{1}{n} \sum_{k=1}^{n} \beta_k = \beta$.
\end{lemma}
\begin{corollary} \label{corollary:cesaro-summable}
Suppose that 
$\lim_{n \to \infty} \sum_{k=0}^{n-1} \alpha_k = \alpha$.
Then, we have
$\lim_{n \to \infty} \sum_{k=0}^{n-1} \frac{n-k}{n} \alpha_k = \alpha$.
\end{corollary}
\begin{proof}
Apply Lemma \ref{lemma:cesaro} to the sequence $\beta_n = \sum_{k=0}^{n-1} \alpha_k$.
\end{proof}

Now, we assume that $X^{n+1}$ obeys the stationary Markov process 
generated by the transition matrix $W_0$,
and denote the variance by $\mathsf{V}$. 
As is shown in \cite[Lemma 6.2]{HW14-1},
$g(X^{n+1}):=\sum_{i=1}^n g(X_{i+1},X_i)$
satisfies
\begin{align*}
\lim_{n \to \infty}
\mathsf{V} [\frac{g^n(X^{n+1})}{\sqrt{n}}]
= \frac{d^2 \phi}{d \theta^2}(0).
\end{align*}
Hence, 
it is enough for Lemma \ref{L20B}
to show that 
\begin{align}
\lim_{n \to \infty}
\mathsf{V} [\frac{g^n(X^{n+1})}{\sqrt{n}}]
=
\mathsf{V} [g(X,X')]
+2 \vec{g}_*^T (Z-A) \vec{g}^*,\Label{5-10}
\end{align}
where
$\vec{g}_*:=[\sum_{x} W(x|\bar{x})g(x,\bar{x})]_{\bar{x}},$ and 
$\vec{g}^*:=[\sum_{\bar{x}} W(x|\bar{x})\tilde{P}_0(\bar{x}) g(x,\bar{x})]_{x}$.

From Proposition \ref{existence-of-z}, $Z = \sum_{n=0}^\infty (W - A)^n$ exists. 
Thus, Corollary \ref{corollary:cesaro-summable} yields that
$Z = \lim_{n\to \infty} \sum_{d=0}^{n-1} \frac{n-d}{n} (W-A)^d$,
which implies
\begin{align}
Z - I = \lim_{n\to \infty} \sum_{d=1}^{n-1} \frac{n-d}{n} (W - A)^d 
= \lim_{n\to \infty} \sum_{d=1}^{n-1} \frac{n-d}{n} (W^d - A),
\Label{eq:proof-3}
\end{align}
where the last equality follows from 
the relation $(W - A)^n = W^n - A$ given in Proposition \ref{existence-of-z}.

By elementary calculation, we have
\begin{align}
\nonumber
\lefteqn{ \frac{1}{n-1} \mathsf{V}\left[ \sum_{k=2}^n  g(X_k,X_{k-1})
 \right] } \\
=& \frac{1}{n-1} \mathsf{E}\Biggl[ \left( 
\sum_{k=2}^n (g(X_k,X_{k-1}) - \mathsf{E} [g(X,X')])
 \right) \nonumber \\
&\quad \cdot 
 \left( 
\sum_{\ell=2}^n ( g(X_\ell,X_{\ell-1})-\mathsf{E} [g(X,X')])
 \right) \Biggr] \nonumber \\
=& \frac{1}{n-1} \sum_{k=2}^n \sum_{\ell=2}^n \left\{ \mathsf{E}\left[ g(X_k,X_{k-1}) g(X_\ell,X_{\ell-1}) \right] - \mathsf{E} [g(X,X')]^2 \right\} \nonumber \\
=& \left\{ \frac{1}{n-1} \sum_{k=2}^n \sum_{\ell=2}^n \mathsf{E}\left[ g(X_k,X_{k-1}) g(X_\ell,X_{\ell-1}) \right]  \right\} 
- (n-1)\mathsf{E} [g(X,X')]^2
 \label{eq:two-favriable-proof-1}
\end{align}
Since
\begin{eqnarray}
\lefteqn{ \Pr\left\{ X_k = x_k, X_{k-1} = x_{k-1}, X_\ell = x_\ell, X_{\ell-1} = x_{\ell-1} \right\} } \nonumber \\
 &=& \left\{
 \begin{array}{ll}
 W(x_k|x_{k-1}) W^{(k - 1 - \ell)}(x_{k-1}|x_\ell) W(x_\ell | x_{\ell-1}) \tilde{P}(x_{\ell -1}) & \mbox{if } k > \ell + 1 \nonumber \\
 W(x_k|x_{k-1}) \delta_{x_{k-1} x_\ell} W(x_\ell | x_{\ell-1}) \tilde{P}(x_{\ell -1}) & \mbox{if } k = \ell + 1 \nonumber \\
 W(x_k|x_{k-1}) \tilde{P}(x_{k-1}) \delta_{x_k x_\ell} \delta_{x_{k-1} x_{\ell -1}} & \mbox{if } k = \ell \nonumber \\
 W(x_\ell | x_{\ell-1}) \delta_{x_{\ell-1} x_k} W(x_k|x_{k-1}) \tilde{P}(x_{k-1}) & \mbox{if } \ell = k+1 \nonumber \\
 W(x_\ell | x_{\ell-1}) W^{(\ell -1 -k)}(x_{\ell -1} | x_{k}) W(x_k | x_{k-1}) \tilde{P}(x_{k-1}) & \mbox{if } \ell > k+1
 \end{array}
 \right.
\end{eqnarray}
and
$\vec{g}_*^T A \vec{g}^* = \mathsf{E} \left[ g(X,X') \right]^2$,
we have
\begin{align*}
&\sum_{k=2}^n \sum_{\ell=2}^n 
\mathsf{E}\left[ g(X_k,X_{k-1}) g(X_\ell,X_{\ell-1}) \right] \\
=&(n-1) \mathsf{E}\left[ g(X,X')^2 \right]
+2(n-2) \vec{g}_*^T (I - A) \vec{g}^*
+2(n-2) \mathsf{E} \left[ g(X,X') \right]^2 \\
& +2 \sum_{k>\ell-1} \vec{g}_*^T (W^{k+1-\ell} - A) \vec{g}^* 
+(n-2)(n-3) \mathsf{E} \left[ g(X,X') \right]^2.
\end{align*}
Thus,
we can rewrite the first term of \eqref{eq:two-favriable-proof-1} as
\begin{align*}
& \hbox{RHS of } \eqref{eq:two-favriable-proof-1}
\nonumber \\
=& \mathsf{V} [g(X,X')]
 + \frac{2}{n-1} \left\{ (n-2) \vec{g}_*^T (I-A) \vec{g}^* + \sum_{k > \ell + 1} \vec{g}_*^T (W^{(k-1 - \ell)} - A) \vec{g}^*  \right\}  \nonumber \\
=& \mathsf{V} [g(X,X')]
 + \frac{2(n-2)}{n-1} \vec{g}_*^T (I-A) \vec{g}^* 
+ \frac{2(n-2)}{n-1} \sum_{d=1}^{n-3} \frac{n-2 -d}{n-2} 
\vec{g}_*^T (W^d - A) \vec{g}^* \nonumber \\
\to & \mathsf{V} [g(X,X')]
 +  2 \vec{g}_*^T (I-A) \vec{g}^* + 2 \vec{g}_*^T(Z-I) \vec{g}^* 
= \mathsf{V} [g(X,X')]+ 2\vec{g}_*^T (Z-A) \vec{g}^*,
\end{align*}
where we used (\ref{eq:proof-3}) with replacing $n$ by $n-3$,
and took the limit $n\to\infty$. 
Combining with \eqref{eq:two-favriable-proof-1}, we 
obtain (\ref{5-10}).
\endproof

\section{Lemmas}
\begin{lemma}\Label{L11-21}
When $f$ is convex
the function $x \mapsto \frac{f(x)}{x}$ with $x\in (0,\infty)$
has the minimum when 
$f'(x)x -f(x)=0$.
In particular, when $f(0)=0$,
the function $x \mapsto \frac{f(x)}{x}$ is monotone increasing for $x\ge 0$.
\end{lemma}
\begin{proof}
We have
\begin{align}
\frac{d}{dx} \frac{f(x)}{x}= \frac{f'(x)x-f(x)}{x^2}.
\end{align}
Since 
\begin{align}
\frac{d}{dx} f'(x)x-f(x)
=f''(x)x \ge 0,
\end{align}
we find that the minimum is realized 
when $f'(x)x-f(x)$.
\end{proof}


\begin{thebibliography}{10}
\bibitem{N}
S. Natarajan, 
``Large deviations, hypotheses testing, and source coding for finite markov chains,'' 
{\em IEEE Trans. Inform. Theory}, Vol. 31, No. 3, 360-365, (1985).

\bibitem{NK}
K. Nakagawa and F. Kanaya, 
``On the converse theorem in statistical hypothesis testing for markov chains,'' 
{\em IEEE Trans. Inform. Theory},
Vol. 39, No. 2, 629-633 (1993).

\bibitem{DZ}
A. Dembo and O. Zeitouni, 
{\em Large Deviations Techniques and Applications, 2nd ed}. Springer (1998).

\bibitem{AN}
S. Amari and H. Nagaoka, 
{\em Methods of Information Geometry}. Oxford University Press (2000).

\bibitem{CLT2}
I. Kontoyiannis and S. P. Meyn,  
``Spectral theory and limit theorems for geometrically ergodic Markov processes,''
{\em The Annals of Applied Probability}, Vol. 13, 304-362 
(2003). 

\bibitem{CLT3}
S. P. Meyn and R. L. Tweedie,  
{\em Markov Chains and Stochastic Stability}, 
Springer-Verlag, London 
(1993).

\bibitem{CLT4}
G. L. Jones,
``On the Markov chain central limit theorem,''
{\em Probability Surveys}, Vol. 1, 299-320 (2004). 

\bibitem{CLT5}
S. Trevezas and N. Limnios,
``Variance estimation in the central limit theorem for Markov chains,''
{\em Journal of Statistical Planning and Inference},
Vol. 139, No. 7, 2242-2253 (2009).

\bibitem{HN}
H. Nagaoka,
``The exponential family of Markov chains and its information geometry''
Proceedings of The 28th Symposium on Information Theory and Its Applications (SITA2005),
Okinawa, Japan, Nov. 20-23, (2005).

\bibitem{Tomamichel}
M. Tomamichel and M. Hayashi, 
``A Hierarchy of Information Quantities for Finite Block Length Analysis
of Quantum Tasks,'' 
{\em IEEE Transactions on Information Theory}, Vol. 59, No. 11, 7693-7710 (2013).


\bibitem{kemeny-snell-book}
J. G. Kemeny and J. L. Snell,
{\em Finite Markov Chains},
Undergraduate Texts in Mathematics,
Springer -Verlag, New York Berlin Heidelberg Tokyo (1960).

\bibitem{CRRao}
C. Radhakrishna Rao,
{\em Linear Statistical Inference and its Applications} 
Wiley Series in Probability and Statistics,
WILEY; second edition (2002).

\bibitem{Csiszar}
I. Csisz\'{a}r, 
``Information-type measures of difference of probability distributions and indirect observation,''
{\em Studia Scientiarum Mathematicarum Hungarica} 2, 229-318 (1967).

\bibitem{Li}
K. Li, 
``Second Order Asymptotics for Quantum Hypothesis Testing,''
arXiv:1208.1400 (2012).
(Accepted in {\em Annals of Statistics}.)

\bibitem{HW14-1}
M. Hayashi and S. Watanabe,
``Information Geometry Approach to Parameter Estimation in Markov Chains,''
arXiv:1401.3814 (2014).

\bibitem{W-H2}
M. Hayashi and S. Watanabe,
``Non-Asymptotic and Asymptotic Analyses on Markov Chains in Several Problems,''
arXiv:1309.7528 (2013)

\bibitem{Strassen}
V. Strassen, 
``Asymptotische Abschatzungen in Shannons Informationstheorie,'' 
{\em In Trans. Third Prague Conf. Inf. Theory}, pages 689-723, Prague, 1962.

\bibitem{Feigin}
P. D. Feigin,
``Conditional Exponential Families and a Representation Theorem for Asympotic Inference,''
{\em The Annals of Statistics}, Vol. 9, No. 3, 597-603 (1981).

\bibitem{Hudson}
I. L. Hudson,
``Large Sample Inference for Markovian Exponential Families with Application to Branching Processes with Immigration,''
{\em Austr. J. Statist.} Vol. 24, 98-112 (1982).

\bibitem{Bhat}
B. R. Bhat, 
``On exponential and curved exponential families in stochastic processes,''
{\em Math. Scientist}, Vol. 13, 121-134 (1988).

\bibitem{Bhat2}
B. R. Bhat,
{\em Stochastic models : analysis and applications},
New Delhi: New Age International (2000).

\bibitem{Stefanov}
V. T. Stefanov,
``Explicit Limit Results for Minimal Sufficient Statistics and Maximum Likelihood Estimators in Some Markov Processes: Exponential Families Approach,''
{\em The Annals of Statistics},
Vol. 23, No. 4, 1073-1101 (1995).

\bibitem{Kuchler-Sorensen}
U. Kuchler and M. Sorensen,
``Exponential Families of Stochastic Processes: A Unifying Semimartingale Approach,''
{\em International Statistical Review},
Vol. 57, No. 2, 123-144 (1989).

\bibitem{Sorensen}
M. Sorensen,
``On Sequential Maximum Likelihood Estimation for Exponential Families of Stochastic Processes,''
{\em International Statistical Review},
Vol. 54, No. 2, 191-210 (1986).

\bibitem{Hoeffding}
W. Hoeffding, 
``Asymptotically optimal tests for multinomial distributions,''
{\em Ann. Math. Statist.}, vol. 36, pp. 369-408 (1965).

\bibitem{Berry}
W. Feller, {\em An Introduction to Probability Theory and Its Applications},
2nd ed. New York, NY, USA: Wiley, (1971).

\bibitem{Herve}
L. Herv\'{e}, J. Ledoux, V. Patilea,
``A uniform Berry--Esseen theorem on M-estimators for geometrically ergodic Markov chains,''
{\em Bernoulli}, Vol. 18, No. 2, 703-734 (2012).

\bibitem{MD}
B. Delyon, A. Juditsky, R. Liptser,
``Moderate deviation principle for ergodic Markov chain. Lipschitz summands,''
{\em From Stochastic Calculus to Mathematical Finance}, 
pp 189-209 (2006).

\bibitem{Massart}
P. Massart,
{\em Concentration Inequalities and Model Selection},
Ecole dfEt\'{e} de Probabilit\'{e}s de Saint-Flour XXXIII, 
Berlin, Springer (2003).

\bibitem{Peskun}
P. H. Peskun,
``Optimum Monte-Carlo sampling using Markov chains,''
{\em Biometrika} Vol. 60, 607-607 (1973).

\bibitem{C1}
K. Latuszynski, B. Miasojedow, and W. Niemiro,
``Nonasymptotic bounds on the mean square error for MCMC estimates via renewal techniques,''
arXiv:1101.5837 (2011).

\bibitem{C2}
R. Adamczak, and W. Bednorz,
``Orlicz integrability of additive functionals of Harris ergodic Markov chains,''
arXiv:1201.3567 (2012).

\bibitem{C3}
R. Adamczak, and W. Bednorz,
``Exponential Concentration Inequalities for Additive Functionals of Markov Chains,''
arXiv:1201.3569 (2012).

\bibitem{MU}
M. Mitzenmacher, and E. Upfal,
{\em Probability and Computing: Randomized Algorithms and Probabilistic Analysis},
Cambridge University Press (2005).




\end{thebibliography}
\end{document}